\documentclass[a4paper,12pt]{amsart}

\usepackage{amsmath}
\usepackage{amsthm}
\usepackage{amssymb}
\usepackage{stmaryrd}
\usepackage{MnSymbol}
\usepackage{bbm}
\usepackage{bm}
\usepackage{dsfont}

\usepackage[top=30truemm, bottom=30truemm, left=25truemm, right=25truemm]{geometry}

\numberwithin{equation}{section}

\newtheorem{theorem}{Theorem}[section]
\newtheorem{lemma}[theorem]{Lemma}

\newtheorem{proposition}[theorem]{Proposition}

\newtheorem{definition}[theorem]{Definition}

\theoremstyle{definition}

\renewcommand{\Re}{\mathrm{Re}\hspace{1pt}}
\renewcommand{\Im}{\mathrm{Im}\hspace{1pt}}

\newcommand{\meas}{\mathrm{meas}} 
\newcommand{\m}{\mathbf{m}}

\title[Value distribution with general shifts]{On the value distribution of the Riemann zeta-function with general shifts}
\author{Keita Nakai}
\date{}

\begin{document}

\begin{abstract}
   This paper studies the value distribution of the Riemann zeta-function under general shifts. 
   We establish discrepancy estimates for the Bohr--Jessen limit theorem for a broad class of shifts. 
   We also prove a universality theorem for logarithmic-power shifts, which are not covered by previous universality results for general shifts. 
   Furthermore, we obtain a quantitative Bohr--Jessen limit theorem for logarithmic-power shifts.

\end{abstract}

\maketitle

\section{Introduction and main results} 

\subsection{Background}

Let $s = \sigma + it$ be a complex variable. 
The Riemann zeta-function is defined by $\zeta(s) = \sum_{n=1}^{\infty} n^{-s}$ for $\sigma > 1$. 
The Riemann zeta-function can be continued meromorphically to the whole complex plane $\mathbb{C}$. 

The study of the value distribution of zeta-functions began with Bohr's denseness results in the 1910s. 
Later, in 1930s, Bohr and Jessen~\cite{BJ, BJ32} refined these results as follows. 

\begin{theorem}[Bohr--Jessen's limit theorem]\label{BJ limit}
    Let $\sigma > 1/2$, let $R$ be any rectangle whose sides are parallel to the coordinate axes, and let 
\[
\mathbb{P}_T(\log\zeta(\sigma + it) \in R) := \frac{1}{T}\meas\left\{t \in [T, 2T] : \log\zeta(\sigma + it) \in R \right\},
\]
where $\meas$ denotes the one-dimensional Lebesgue measure. 
Then the limit 
\[
\lim_{T \to \infty} \mathbb{P}_T(\log\zeta(\sigma + it) \in R) = W(\sigma, R)
\]
exists.
\end{theorem}

Note that the above formulation is slightly different from the original statement in~\cite{BJ, BJ32}. 
In other words, $W(\sigma, R)$ represents the probability that $\log\zeta(\sigma + it)$ belongs to $R$ when $t$ is chosen randomly. 
Bohr and Jessen also proved that $W(\sigma, R)$ can be written as
\[
W(\sigma, R)=\iint_R F(z)\,du\,dv,
\]
where $F(z)$ is a real-valued, bounded, continuous, and nonnegative function for $1/2 < \sigma \le 1$. 
Moreover, it is known that $W(\sigma, R)$ can be written as 
\[
W(\sigma, R) = \mathbb{P}(\log\zeta(\sigma, X) \in R),
\]
where $\zeta(\sigma, X)$ is a random Euler product attached to the Riemann zeta-function, that is 
\[
\zeta(s, X) :=\prod_{p} \left(1-X(p)p^{-s} \right)^{-1}.
\]
Here, $X = \{X(p)\}_{p}$ is a sequence of independent random variables (we define it in Section~\ref{Pre}). 

We are also interested in the discrepancy estimate
\[
\left|\mathbb{P}_T(\log \zeta(\sigma + it) \in R) - \mathbb{P}(\log\zeta(\sigma, X) \in R)\right|.
\]
In 1980s, Matsumoto~\cite{Ma85, Ma87, Ma88} first evaluated some discrepancy estimates for this quantity.  
Since then, the discrepancy bound has been improved by several authors. 
The best bound currently known was obtained by Lamzouri--Lester--Radziwi\l\l~\cite{LLR}. 

On the other hand, Voronin extended Bohr's denseness theorem to multidimensional denseness results~\cite{Vo72}, and then established the celebrated universality theorem for the Riemann zeta-function~\cite{Vo75}. 
A modern formulation of universality for the Riemann zeta-function is as follows.

\begin{theorem}[Universality theorem]
Let $\mathcal{K}$ be a compact set in the strip $1/2 < \sigma < 1$ with connected complement, and let $f(s)$ be a non-vanishing continuous function on $\mathcal{K}$ that is analytic in the interior of $\mathcal{K}$. Then, for any $\varepsilon > 0$,
\[
\liminf_{T \to \infty} \frac{1}{T} \meas \left\{\tau \in [T, 2T] : \sup_{s \in \mathcal{K}} |\zeta(s + i\tau) - f(s)| < \varepsilon \right\} > 0.
\]
\end{theorem}

Roughly speaking, any non-vanishing holomorphic function can be approximated by the Riemann zeta-function with a suitable vertical shift $i\tau$, and the set of such shifts has positive lower density. 

Bagchi~\cite{Ba} proved universality by using probabilistic method. 
Nowadays, universality has been generalized and extended to a wide variety of zeta-functions and $L$-functions (see, for example,~\cite{Ma15}). 
One such generalization is to replace $i\tau$ by $i\gamma(\tau)$, where $\gamma(\tau)$ is a real-valued function. 
In other words, for real-valued functions $\gamma_j(\tau)$, we are interested in the question whether the universality property for some $L$-function
\begin{equation}
\liminf_{T \to \infty} \frac{1}{T}\meas\left\{\tau \in [T, 2T] : \max_{1 \le j \le r} \sup_{s \in K} |L(s + i\gamma_j(\tau)) - f(s)| < \varepsilon \right\} > 0
\label{joint universality}
\end{equation}
holds for any $\varepsilon > 0$ or not. 

The universality property of type~(\ref{joint universality}) was first established by Nakamura~\cite{Na09}. 
It was the first result extending joint universality for Dirichlet $L$-functions to the shifts $\gamma_j(\tau) = a_j \tau$, where $a_1, \dots, a_r$ are real algebraic numbers linearly independent over $\mathbb{Q}$.
Moreover, when $r = 2$, further results were obtained in~\cite{Na09}, \cite{Pa09}, and \cite{Pa16}. 
Furthermore, Pa\'{n}kowski~\cite{Pa18} proved that the universality property~(\ref{joint universality}) holds when $\gamma_j(\tau) = \alpha_j \tau^{a_j} (\log \tau)^{b_j}$, where $\alpha_j$, $a_j$, and $b_j$ satisfy certain assumptions. 
Laurin\v{c}ikas and \v{S}iau\v{c}i\={u}nas~\cite{LS} obtained an affirmative answer to the above question for shifts $\gamma_j(\tau) = t_{\tau}^{\alpha_j}$, where $t_\tau$ is the Gram function and $\alpha_1, \dots, \alpha_r$ are distinct fixed positive numbers. 

Besides explicitly defined shifts, in \cite{LMS19}, \cite{La21Rama}, \cite{La21Math}, and \cite{La22}, several variants of universality theorems for the Riemann zeta-function were established for axiomatically defined shifts such as polynomials. 
Furthermore, the author~\cite{Na} introduced the class $\mathcal{F}$, and proved joint universality for the Riemann zeta-function with more general shifts, including exponential shifts and shifts with faster growth.

\subsection{Main theorems}

In this paper, we first investigate how discrepancy estimates in the Bohr--Jessen limit theorem change when the classical shift $+it$ is replaced by a general shift $+i\gamma(t)$.

To state our first main theorem, we introduce the class $\mathcal{F}$ introduced in \cite{Na}.

\begin{definition}
Let $\mathcal{F}$ denote the class of functions $\gamma$ satisfying the following assumptions.

\begin{enumerate}

\item[(F1)] There exist positive real numbers $T_0>0$ and $T_1>0$ such that $\gamma:[T_0,\infty)\to[T_1,\infty)$ is strictly increasing. Moreover, $\gamma(T)\to\infty$ as $T\to\infty$.

\item[(F2)] The function $\gamma$ is continuously differentiable, and $\gamma'$ is monotonic on $[T_0,\infty)$.

\item[(F3)] There exists a positive constant $\alpha$ such that $\alpha \gamma(\tau)\le \tau\gamma'(\tau)$
for all $\tau\ge T_0$.

\end{enumerate}

\end{definition}
For the discrepancy problem, we require one additional growth condition.
Accordingly, we introduce the subclass $\mathcal{F}'$. 
A function $\gamma \in \mathcal{F}$ belongs to $\mathcal{F}'$ if $\gamma$ satisfies the following condition:
\begin{enumerate}
    \item[(F4)] For any $\varepsilon > 0$ and $c \ge 1$,
    \[
    \log \gamma(cT) \ll_{\varepsilon} \gamma(T)^{\varepsilon}.
    \]
\end{enumerate}

Typical examples of functions in $\mathcal{F}$ include
$t^a(\log t)^b$ $(a>0,\ b\in\mathbb{R})$,
exponential functions, the Gamma function, and the $n$th iterated exponential function $(n\in\mathbb{N})$.
All of these examples also belong to $\mathcal{F}'$.
On the other hand, Proposition~\ref{gamma c} shows that the inclusion $\mathcal{F}' \subsetneq \mathcal{F}$ is proper.

Our first main theorem is the following.

\begin{theorem}\label{main general}
Let $1/2<\sigma<1$ and $\gamma\in\mathcal{F}'$. Define
\[
D_{\sigma,\gamma}(T)
:=
\sup_{\mathcal R}
\left|
\mathbb{P}_T(\log\zeta(\sigma+i\gamma(t))\in\mathcal R)
- \mathbb{P}(\log\zeta(\sigma,X)\in\mathcal R)
\right|,
\]
where $\mathcal R$ ranges over all rectangles in $\mathbb{C}$ whose sides are parallel to the coordinate axes.
Then
\[
D_{\sigma,\gamma}(T)
\ll
\frac1{(\log\gamma(T))^\sigma}
\]
for sufficiently large $T > 0$.
\end{theorem}

For example, when $\gamma(t)=e^t$, Theorem~\ref{main general} yields
\[
D_{\sigma,\gamma}(T)
\ll
T^{-\sigma}.
\]

By \cite[Proposition~1.2]{LLR}, for fixed $\varepsilon > 0$, $D_{\sigma, t}(T) = \Omega(T^{1-2\sigma -\varepsilon}) $ holds. 
Comparing this with the classical case $\gamma(t)=t$, we see that $\mathbb P_T(\log\zeta(\sigma+ie^t)\in\mathcal R)$
converges to $\mathbb P(\log\zeta(\sigma,X)\in\mathcal R)$
faster than $\mathbb P_T(\log\zeta(\sigma+it)\in\mathcal R)$. 
Theorem 1.4 provides a unified discrepancy estimate for the Bohr--Jessen limit theorem under a broad class of shifts, encompassing polynomial, exponential, and many other naturally occurring shifts.

On the other hand, all known results are restricted to shifts whose growth is either of polynomial type or faster than any polynomial. 
In particular, they do not cover logarithmic shifts such as
\[
\gamma(\tau)=(\log\tau)^\alpha.
\]
Since logarithmic-power shifts do not belong to the class
$\mathcal F$, Theorem~\ref{main general} cannot be applied to them.
Our next result shows that universality nevertheless holds for this class of shifts. 
To describe the second main result, we introduce some notations. 
We define the number of zeros of the Riemann zeta-function by
\[
N(\sigma, T) := \# \{\rho \in \mathbb{C} : \zeta(\rho) = 0,\ \Re(\rho) > \sigma,\ 0 < \Im(\rho) < T \}.
\]
Suppose that
\[
N(\sigma,T) \ll_{\varepsilon} T^{\Phi(\sigma)+\varepsilon},
\]
where $\Phi(\sigma)$ is a nonnegative non-increasing function for $\sigma>1/2$. For $\alpha>1$, define
\[
x_{\Phi}(\alpha)
:=
\inf\left\{x>\frac12:\Phi(x)<1-\frac1\alpha\right\}.
\]

Note that, by combining Ingham's result~\cite{In} and the result of Guth--Maynard~\cite{GM}, we can take
\[
 x_\Phi(\alpha) = \min\left\{\frac{\alpha+2}{2\alpha+1}, \frac{17}{30}+\frac{13}{30\alpha}\right\},
\]
and under the density hypothesis, we can take
\[
x_\Phi(\alpha) = \frac12+\frac{1}{2\alpha}.
\]
Then, we establish a joint universality theorem with logarithmic-power shifts.

\begin{theorem}\label{main1}
Assume that $1 < \alpha_1 < \dots < \alpha_r$.
Let $\mathcal{K}$ be a compact set in the strip $x_\Phi(\alpha_1) < \sigma < 1$ with connected complement, and let $f_1(s), \dots, f_r(s)$ be continuous functions on $\mathcal{K}$ that are analytic in the interior of $\mathcal{K}$. Then, for any $\varepsilon > 0$,
\[
\liminf_{T \to \infty} \frac{1}{T} \meas \left\{\tau \in [T, 2T] : \max_{1 \le j \le r} \sup_{s \in \mathcal{K}} |\log\zeta(s + i(\log \tau)^{\alpha_j}) - f_j(s)| < \varepsilon \right\} > 0.
\]
In particular, under the Riemann hypothesis, the above universality result can be extended to $1/2 < \sigma < 1$. 
\end{theorem}

Note that our method does not allow us to take $\alpha_1 = 1$, and therefore we are unable to determine whether universality holds for the shift $\gamma(\tau)=\log\tau$ even if we assume the Riemann hypothesis.

Motivated by Theorem~\ref{main1}, we investigate the value distribution of
$\log\zeta(\sigma+i(\log t)^r)$ through the Bohr--Jessen type limit theorem. 
By similar arguments, we can show that
$\mathbb{P}_T(\log\zeta(\sigma + i(\log t)^r) \in \mathcal{R})$
converges to
$\mathbb{P}(\log\zeta(\sigma, X) \in \mathcal{R})$
as $T \to \infty$ for $r > 1$. 
We next establish a more precise limit theorem. 
Our third main theorem is as follows.

\begin{theorem}\label{main2}
Let $r > 1$ and $x_\Phi(r) < \sigma < 1$.
Put
\[
D_{\sigma, r}(T)
:=
\sup_{\mathcal{R}}
\left|
\mathbb{P}_T(\log\zeta(\sigma + i(\log t)^r) \in \mathcal{R})
-
\mathbb{P}(\log\zeta(\sigma, X) \in \mathcal{R})
\right|,
\]
where $\mathcal{R}$ runs through all rectangles in $\mathbb{C}$ whose sides are parallel to the coordinate axes.
Then, we have
\[
D_{\sigma ,r}(T)
\ll
\frac{1}{((r-1)\log\log T)^\sigma}
\]
for sufficiently large $T  > 0$.
In particular, assuming the Riemann hypothesis, we can prove the above estimate for $1/2 < \sigma < 1$. 
\end{theorem}

The discrepancy estimate obtained above becomes weaker as $r \to 1$, which illustrates the difficulty of the case $r=1$. 
Moreover, the difficulty of the case $\log t$ will be discussed in Section~\ref{conclusion}. 

\section{Preliminaries for Probability settings} \label{Pre}

We write $\mathcal{B}(T)$ for the set of all Borel subsets of $T$ which is a topological space. 
Let $S^1 = \{ s \in \mathbb{C} : |s| = 1 \}$. 
For each prime $p$, let $S_p = S^1$ and  $\Omega = \prod_p S_p$. 
Define $\Omega^r = \Omega_1 \times \dots \times\Omega_r$, where $\Omega_j = \Omega$ for all $j = 1, \dots, r$. 
Then there exists the probability Haar measure $\mathbf{m}^r$ on $(\Omega^r, \mathcal{B}(\Omega^r))$. 
We note that $\mathbf{m}^r$ is written by $\mathbf{m}^r = \otimes_{j=1}^{r} \mathbf{m}_j$, where $\mathbf{m}_j = \otimes_{p} \mathbf{m}_p$ is the probability Haar measure on $(\Omega_j, \mathcal{B}(\Omega_j))$ and $\mathbf{m}_p$ is the probability Haar measure on $(S_p, \mathcal{B}(S_p))$. 
The expectation with respect to $\mathbf m^r$ is defined by $\mathbb{E}^{\mathbf{m}^r}[f] = \int_{\Omega^r} f \, d\mathbf{m}^r$, and for simplicity, we write $\mathbb{E}[f] := \mathbb{E}^{\mathbf{m}}[f]$. 

Let $X(p)$ be the projection of $X\in \Omega$ to the coordinate space $S_p$. 
For $X \in \Omega$, define $X(1) :=1$, 
\[
X(n) := \prod_p X(p)^{\nu(n ; p)}, 
\]
where $\nu(n ; p)$ is the exponent of the prime $p$ in the prime factorization of $n$. 
Similarly, for the projection $X_j \in \Omega_j$, we define $X_j(n)$. 
We define a random Euler product
\begin{align*}
\zeta(s, X) &= \prod_{p} \left(1 - \frac{X(p)}{p^s} \right)^{-1} \\
              &= \sum_{n=1}^{\infty} \frac{X(n)}{n^s}
\end{align*}
and the random element
\[
\underline{\zeta}(s, X) = (\zeta(s, X_1), \dots, \zeta(s, X_r)).
\]

\section{Dirichlet polynomials with general shifts} \label{section3}

In this section, we establish the following proposition and lemmas, which will be used in the proof of Theorem~\ref{main general}.
Define
\[
P_Y(\sigma + it) = \sum_{2 \le n \le Y} \frac{\Lambda(n)}{n^{\sigma + it} \log n}, \quad P_Y(\sigma, X) = \sum_{2 \le n \le Y} \frac{\Lambda(n) X(n)}{n^\sigma \log n}, 
\]
where $Y > 2$, $X = \{X(n)\}_{n \ge 1}$ is a sequence of random variables defined in Section~\ref{Pre}. 
Define
\[
\tilde{\gamma}(T) = \min\{\gamma'(T), \gamma'(2T) \}. 
\]
In this section, we prove the next proposition. 

\begin{proposition} \label{prop1 general}
Let $1/2 < \sigma < 1$. 
Let $T, V$ be large. 
Define 
\[
A_T = A_T(V, Y, \sigma, \gamma) =  \{t \in [T, 2T] : |P_Y(\sigma + i\gamma(t))| \le V  \}
\]
for $Y \ge 3$. 
If 
\begin{equation} \label{Y range general}
3 \le Y \le \exp\left(\frac{\log T \tilde{\gamma}(T)}{V^{\frac{1}{1-\sigma}} (\log V)^{\frac{\sigma}{1-\sigma}}} \right)
\end{equation}
holds and $\gamma$ satisfies (F1) and (F2), then there exist constants $d_1 = d_1(\sigma) > 0$ and $d_2 = d_2(\sigma) >0$ such that for any complex numbers $z_1$, $z_2$ with $|z_1|, |z_2| \le d_1(V\log V)^\frac{\sigma}{1-\sigma}$ we have 
\begin{align*}
&\frac{1}{T} \int_{A_T} \exp\left(z_1P_Y(\sigma + i\gamma(t)) + z_2 \overline{P_Y(\sigma + i\gamma(t))}\right)\, dt \\
&= \mathbb{E}\left[\exp\left(z_1P_Y(\sigma, X) + z_2 \overline{P_Y(\sigma, X)}\right)\right] + E, 
\end{align*}
where $E$ is bounded by 
\begin{align*}
    E \ll \frac{1}{T \tilde{\gamma}(T)} \left( (V\log V)^{\frac{\sigma}{1-\sigma}} Y\right)^{V^\frac{1}{1-\sigma}(\log V)^{\frac{\sigma}{1-\sigma}}} + \exp\left(-d_2 V^{\frac{1}{1-\sigma}}(\log V)^{\frac{\sigma}{1-\sigma}} \right).
\end{align*}
\end{proposition}

First, we establish a moment of Dirichlet polynomials with general shifts.

\begin{lemma} \label{mean general}
    Let $1/2 < \sigma < 1$. 
    Let $\gamma$ satisfy (F1) and (F2), $T$ be large, and $Y \ge 3$. 
    Then, for $k, l \in \mathbb{N}$, we have
    \begin{align*}
        &\frac{1}{T}\int_{T}^{2T} P_Y(\sigma + i\gamma(t))^k\overline{P_Y(\sigma + i\gamma(t))}^l\, dt \\
        &= \mathbb{E}\left[P_Y(\sigma, X)^k \overline{P_Y(\sigma, X)}^l\right] + O\left(\frac{Y^{2(k+l)}}{T\tilde{\gamma}(T)} \right). 
    \end{align*}

\end{lemma}

\begin{proof}
    By simple arguments, we see that 
    \begin{align*}
        &\frac{1}{T}\int_{T}^{2T} P_Y(\sigma + i\gamma(t))^k\overline{P_Y(\sigma + i\gamma(t))}^l\, dt \\ 
        &= \frac{1}{T} \sum_{\substack{p_1^{a_1}, \dots, p_k^{a_k} \le Y \\ q_1^{b_1}, \dots, q_l^{b_l} \le Y}} \frac{1}{a_1p_1^{a_1} \dots a_kp_k^{a_k} b_1q_1^{b_1} \dots b_lq_l^{b_l}}\int_{T}^{2T} \left(\frac{q_1^{b_1} \dots q_l^{b_l}}{p_1^{a_1} \dots p_k^{a_k}}\right)^{i\gamma(t)}\, dt. 
    \end{align*}
When $p_1^{a_1} \dots p_k^{a_k} \ne q_1^{b_1} \dots q_l^{b_l}$, applying the first derivative test, we obtain 
\begin{align*}
    \int_{T}^{2T} \left(\frac{q_1^{b_1} \dots q_l^{b_l}}{p_1^{a_1} \dots p_k^{a_k}}\right)^{i\gamma(t)}\, dt 
    \ll \frac{1}{\tilde\gamma(T)} Y^{k+l}
\end{align*}
    since for $p_1^{a_1}, \dots, p_k^{a_k} \le Y$, $q_1^{b_1}, \dots, q_l^{b_l} \le Y$, 
    \[
    \frac{1}{\left|\log\frac{q_1^{b_1} \dots q_l^{b_l}}{p_1^{a_1} \dots p_k^{a_k}}\right|} \ll Y^{k+l}
    \]
    holds.
    Therefore, the non-diagonal term can be estimated as 
    \begin{align*}
        &\frac{1}{T} \sum_{\substack{p_1^{a_1}, \dots, p_k^{a_k} \le Y \\ q_1^{b_1}, \dots, q_l^{b_l} \le Y \\ p_1^{a_1} \dots p_k^{a_k} \ne q_1^{b_1}, \dots, q_l^{b_l}}} \frac{1}{a_1p_1^{a_1} \dots a_kp_k^{a_k} b_1q_1^{b_1} \dots b_lq_l^{b_l}}\int_{T}^{2T} \left(\frac{q_1^{b_1} \dots q_l^{b_l}}{p_1^{a_1} \dots p_k^{a_k}}\right)^{i\gamma(t)} dt\\
        &\ll \frac{Y^{k+l}}{T\tilde\gamma(T)}\left(\sum_{p^a \le Y} \frac{1}{ap^{a\sigma}}\right)^{k+l} \ll \frac{Y^{2(k+l)}}{T\tilde\gamma(T)}. 
    \end{align*}
    The diagonal contribution is  
    \begin{align*}
     &\frac{1}{T}\sum_{\substack{p_1^{a_1}, \dots, p_k^{a_k} \le Y \\ q_1^{b_1}, \dots, q_l^{b_l} \le Y \\ p_1^{a_1} \dots p_k^{a_k} = q_1^{b_1}, \dots, q_l^{b_l}}} \frac{1}{a_1p_1^{a_1} \dots a_kp_k^{a_k} b_1q_1^{b_1} \dots b_lq_l^{b_l}} \int_{T}^{2T} \left(\frac{q_1^{b_1} \dots q_l^{b_l}}{p_1^{a_1} \dots p_k^{a_k}}\right)^{i\gamma(t)} dt\\
     &= \sum_{\substack{p_1^{a_1}, \dots, p_k^{a_k} \le Y \\ q_1^{b_1}, \dots, q_l^{b_l} \le Y \\ p_1^{a_1} \dots p_k^{a_k} = q_1^{b_1}, \dots, q_l^{b_l}}} \frac{1}{a_1p_1^{a_1} \dots a_kp_k^{a_k} b_1q_1^{b_1} \dots b_lq_l^{b_l}} \\
     &=  \mathbb{E}\left[P_Y(\sigma, X)^ k\overline{P_Y(\sigma, X)}^l\right]
    \end{align*}
    Therefore, we obtain the conclusion. 
\end{proof}

Next, we show a lemma for higher moments of Dirichlet polynomials with $\gamma(t)$. 

\begin{lemma} \label{higher moment general}
    Let $\{a(p)\}$ be any complex sequence. 
    Let $\gamma$ satisfy (F1) and (F2), $T$ be large, and $Y \ge 3$. 
    For $k \ge 1$ with 
    \begin{equation} \label{Y}
    Y \le \left(\frac{T\tilde\gamma(T)}{\log T\tilde\gamma(T)} \right)^{\frac{1}{k}}, 
    \end{equation}
    we have
    \[
    \frac{1}{T} \int_{T}^{2T} \left|\sum_{p \le Y}a(p)p^{-i\gamma(t)} \right|^{2k}\, dt \ll k! \left(\sum_{p \le Y} |a(p)|^2 \right)^k.
    \]
\end{lemma}

\begin{proof}
   The proof follows the argument of \cite[Lemma~3]{So}. 
    Write 
    \begin{equation*}
   a_{k, Y}(n)=
  \begin{cases}
     \binom{k}{\alpha_1, \dots, \alpha_N} \prod_{j=1}^ra(p_j)^{\alpha_j} & \text{if  $n = \prod_{j=1}^Np_j^{\alpha_j}$, $p_j \le Y$, $\alpha_j \ge 1$,} \\
       0   & \text{otherwise.}
  \end{cases}
\end{equation*}
Then, 
\begin{align*}
    \int_{T}^{2T} \left|\sum_{p \le Y}a(p)p^{-i\gamma(t)} \right|^{2k}\, dt
    &= \sum_{m, n \le Y^k} a_{k,Y}(m)\overline{a_{k, Y}(n)}\int_{T}^{2T} \left(\frac{n}{m} \right)^{i\gamma(t)}dt \\ 
    &= T\sum_{n\le Y^k}|a_{k,Y}(n)|^2 + O\left(\frac{1}{\tilde\gamma(T)}\sum_{\substack{m, n \le Y^k \\ m \ne n}} \frac{|a_{k, Y}(m)a_{k,Y}(n)|}{\left|\log \frac{n}{m} \right|} \right)
\end{align*}
holds. 
From \cite[Lemma~3]{So}, we have
\[
\frac{1}{\tilde\gamma(T)}\sum_{\substack{m, n \le Y^k \\ m \ne n}} \frac{|a_{k, Y}(m)a_{k,Y}(n)|}{\left|\log \frac{n}{m} \right|} \ll \frac{1}{\tilde\gamma(T)} Y^k\log Y^k \sum_{n \le Y^k}|a_{k, Y}(n)|^2. 
\]
By the assumption
\[
    Y^k \le \frac{T\tilde\gamma(T)}{\log T\tilde\gamma(T)},  
    \]
we see that 
\[
Y^k \log Y^k \le T\tilde\gamma(T)
\]
holds.
Hence, 
\begin{align*}
    \int_{T}^{2T} \left|\sum_{p \le Y}a(p)p^{-i\gamma(t)} \right|^{2k}\, dt \ll T \sum_{n \le Y^k} |a_{k, Y}(n)|^2 \ll T k!\left(\sum_{p \le Y} |a(p)|^2\right)^k. 
\end{align*}
\end{proof}

The following two lemmas are immediate consequences.

\begin{lemma} \label{higher P general}
    Let $1/2 < \sigma < 1$. 
    Let $\gamma$ satisfy (F1) and (F2), and $T, V$ be large. 
    Then, for $k \in \mathbb{N}$ satisfying \eqref{Y}, there exists a positive constant $C = C(\sigma)$ such that 
    \[
    \frac{1}{T}\int_{T}^{2T} |P_Y(\sigma + i\gamma(t))|^{2k}\,dt \le \left(\frac{Ck^{1-\sigma}}{(\log2k)^\sigma} \right)^{2k}. 
    \]
\end{lemma}

\begin{proof}
    The lemma follows from Lemma~\ref{higher moment general} in the same way as \cite[Lemma~2.4]{EIM}.
\end{proof}

\begin{lemma} \label{large deviation general}
    Let $1/2 < \sigma < 1$. 
    Let $\gamma$ satisfy (F1) and (F2), and $T, V$ be large. 
    If $Y$ satisfies \eqref{Y range general}, then there exists a positive constant $c_1 = c_1(\sigma)$ such that 
    \[
    \mathbb{P}_T(|P_Y(\sigma + i\gamma(t))| \ge V) \le \exp\left(-c_1V^{\frac{1}{1-\sigma}}(\log V)^{\frac{\sigma}{1-\sigma}} \right). 
    \]
\end{lemma}

\begin{proof}
    Combining Lemma~\ref{higher P general} with the proof of \cite[Lemma~2.5]{EIM}, we can prove this lemma.
    \end{proof}

\begin{proof}[Proof of Proposition~\ref{prop1 general}]
    Let $Z = c_2 V^{\frac{1}{1-\sigma}}(\log V)^{\frac{\sigma}{1-\sigma}}$, 
    where $c_2$ is a sufficiently small positive constant. 
    By the definition of $A_T$, we can divide the integral as follows:
    \begin{align*}
        &\int_{A_T} \exp(z_1P_Y(\sigma + i\gamma(t)) + z_2 \overline{P_Y(\sigma + i\gamma(t))})\,dt \\
        &= \sum_{\substack{k + l < Z \\ k, l \in \mathbb{Z}_{\ge 0}}} \frac{z_1^k z_2^l}{k!l!}\int_{A_T} P_Y(\sigma + i\gamma(t))^k \overline{P_Y(\sigma + i\gamma(t))}^l\,dt 
        + O\left(T \sum_{\substack{k + l \ge Z \\ k, l \in \mathbb{Z}_{\ge 0}}} \frac{|z_1|^k|z_2|^l}{k!l!} V^{k + l} \right). 
    \end{align*}

    If $|z_1|, |z_2| \le 2^{-1}e^{-2}c_2 (V\log V)^{\frac{\sigma}{1-\sigma}} = 2^{-1}e^{-2}V^{-1}Z$ holds, then we have
    \[
    \sum_{\substack{k + l \ge Z \\ k, l \in \mathbb{Z}_{\ge 0}}} \frac{|z_1|^k|z_2|^l}{k!l!} V^{k + l} \ll \exp\left(-c_2V^{\frac{1}{1-\sigma}} (\log V)^{\frac{\sigma}{1-\sigma}}\right). 
    \]
    On the other hand, since for sufficiently small $c_2$ and $k + l < Z$, 
    \[
    Y^k \le \frac{T\tilde\gamma(T)}{\log T\tilde\gamma(T)}
    \]
    holds, for $1 \le k + l \le Z$ we have 
    \begin{align*}
        &\frac{1}{T} \int_{[T, 2T]\setminus A_T} P_Y(\sigma + i\gamma(t)) \overline{P_Y(\sigma + i\gamma(t))}^ldt \\
        &\le \left(\frac{1}{T}\meas([T, 2T] \setminus A_T) \right)^{\frac{1}{2}} \left(\frac{1}{T}\int_{T}^{2T} |P_Y(\sigma + i\gamma(t))|^{2(k+l)}dt \right)^{\frac{1}{2}} \\
        &\ll \exp\left(-\frac{c_1}{2} V^{\frac{1}{1-\sigma} }(\log V)^{\frac{\sigma}{1-\sigma}}\right)\left(\frac{C(\sigma)(k+l)^{1-\sigma}}{(\log2(k+l))^{\sigma}} \right)^{k+l} \\
        &\le  \exp\left(-\frac{c_1}{2} V^{\frac{1}{1-\sigma} }(\log V)^{\frac{\sigma}{1-\sigma}}\right)\left(\frac{C(\sigma)Z^{1-\sigma}}{(\log2Z)^{\sigma}} \right)^{k+l}
    \end{align*}
    by using Lemma~\ref{higher P general} and Lemma~\ref{large deviation general}. 
    When $k + l = 0$, one can obtain the same bound. 
    Hence, we have
    \begin{align*}
        &\frac{1}{T}\sum_{\substack{k + l < Z \\ k, l \in \mathbb{Z}_{\ge 0}}} \frac{z_1^k z_2^l}{k!l!}\int_{A_T} P_Y(\sigma + i\gamma(t))^k\overline{P_Y(\sigma + i\gamma(t))}^ldt \\
        &\ll \exp\left(-\frac{c_1}{2} V^{\frac{1}{1-\sigma} }(\log V)^{\frac{\sigma}{1-\sigma}}\right)\sum_{\substack{0\le k+l < Z \\ k, l \in \mathbb{Z}_{\ge 0}}} \frac{1}{k!l!} \left(2^{-1}C'c_2 V^{\frac{1}{1-\sigma}}(\log V)^{\frac{\sigma}{1-\sigma}} \right)^{k+l} \\
        &\ll \exp\left(-\frac{c_1}{2} V^{\frac{1}{1-\sigma} }(\log V)^{\frac{\sigma}{1-\sigma}}\right) \exp\left(C'c_2 V^{\frac{1}{1-\sigma} }(\log V)^{\frac{\sigma}{1-\sigma}}\right), 
    \end{align*}
    where $C'$ is a positive constant which is independent of $V$ and $c_2$. 
    Here, choosing $c_2 \le \frac{c_1}{C'}$, we deduce 
    \begin{align*}
        &\frac{1}{T}\sum_{\substack{k + l < Z \\ k, l \in \mathbb{Z}_{\ge 0}}} \frac{z_1^k z_2^l}{k!l!}\int_{A_T} P_Y(\sigma + i\gamma(t))^k\overline{P_Y(\sigma + i\gamma(t))}^ldt \\
        &\ll \exp\left(-\frac{c_1}{4} V^{\frac{1}{1-\sigma} }(\log V)^{\frac{\sigma}{1-\sigma}}\right) 
        \ll \exp\left(-c_2 V^{\frac{1}{1-\sigma} }(\log V)^{\frac{\sigma}{1-\sigma}}\right). 
    \end{align*}
    Combining these estimates, the following holds:  
    \begin{align*}
        &\frac{1}{T}\int_{A_T} \exp(z_1P_Y(\sigma + i\gamma(t)) + z_2 \overline{P_Y(\sigma + i\gamma(t))})dt \\
        &= \frac{1}{T}\sum_{\substack{k + l < Z \\ k, l \in \mathbb{Z}_{\ge 0}}} \frac{z_1^k z_2^l}{k!l!}\int_{T}^{2T} P_Y(\sigma + i\gamma(t))^k\overline{P_Y(\sigma + i\gamma(t))}^ldt \\
        & \quad- \frac{1}{T}\sum_{\substack{k + l < Z \\ k, l \in \mathbb{Z}_{\ge 0}}} \frac{z_1^k z_2^l}{k!l!}\int_{[T,2T] \setminus A_T} P_Y(\sigma + i\gamma(t))^k\overline{P_Y(\sigma + i\gamma(t))}^ldt \\
        &\quad + O\left(\exp\left(-c_2V^{\frac{1}{1-\sigma}} (\log V)^{\frac{\sigma}{1-\sigma}}\right) \right) \\
        &= \frac{1}{T}\sum_{\substack{k + l < Z \\ k, l \in \mathbb{Z}_{\ge 0}}} \frac{z_1^k z_2^l}{k!l!}\int_{T}^{2T} P_Y(\sigma + i\gamma(t))^k\overline{P_Y(\sigma + i\gamma(t))}^ldt + O\left(\exp\left(-c_2V^{\frac{1}{1-\sigma}} (\log V)^{\frac{\sigma}{1-\sigma}}\right) \right). 
    \end{align*}
    Applying Lemma~\ref{mean general}, we have
    \begin{align*}
        &\frac{1}{T}\sum_{\substack{k + l < Z \\ k, l \in \mathbb{Z}_{\ge 0}}} \frac{z_1^k z_2^l}{k!l!}\int_{T}^{2T} P_Y(\sigma + i\gamma(t))^k\overline{P_Y(\sigma + i\gamma(t))}^ldt \\
        &= \sum_{\substack{k + l < Z \\ k, l \in \mathbb{Z}_{\ge 0}}} \frac{z_1^k z_2^l}{k!l!}\mathbb{E}\left[ P_Y(\sigma, X)^k \overline{P_Y(\sigma, X)}^l\right] + O\left(\frac{1}{T\tilde\gamma(T)} \sum_{\substack{k+l < Z \\ k, l \ge 0}} \frac{|z_1|^k|z_2|^l}{k!l!} Y^{2(k+l)}\right).
    \end{align*}
    Moreover, the above error term is bounded by 
    \begin{align*}
        \ll \frac{1}{T\tilde\gamma(T)} \left((V\log V)^{\frac{\sigma}{1-\sigma}}Y \right)^{2Z}.
    \end{align*}
    Finally, by similar arguments, we see that 
    \begin{align*}
        &\sum_{\substack{k + l < Z \\ k, l \in \mathbb{Z}_{\ge 0}}} \frac{z_1^k z_2^l}{k!l!}\mathbb{E}\left[ P_Y(\sigma, X)^k \overline{P_Y(\sigma, X)}^l\right] \\
        &= \mathbb{E}\left[\exp\left(z_1P_Y(\sigma, X) + z_2\overline{P_Y(\sigma, X)} \right)\right] - \sum_{\substack{k + l \ge Z \\ k, l \in \mathbb{Z}_{\ge 0}}} \frac{z_1^k z_2^l}{k!l!}\mathbb{E}\left[ P_Y(\sigma, X)^k \overline{P_Y(\sigma, X)}^l\right] \\
        &=\mathbb{E}\left[\exp\left(z_1P_Y(\sigma, X) + z_2\overline{P_Y(\sigma, X)} \right)\right] + O\left(\exp\left(-c_2V^{\frac{1}{1-\sigma}} (\log V)^{\frac{\sigma}{1-\sigma}}\right) \right). 
    \end{align*}
This completes the proof. 
\end{proof}

\section{Proof of Theorem~\ref{main general}} \label{proof main general}

Granville and Soundararajan~\cite{GS} proved that $\log \zeta(s)$ can be approximated by certain Dirichlet polynomials. 
Subsequently, Endo~\cite{En} also showed the following result. 

\begin{lemma} [{\cite[Lemma~3]{En}}]\label{appro}
    Let $1/2 < \sigma_0 \leq 1$ be fixed and let $T$ and $y$ be large numbers with $T \geq y + 3$.
Put 
\begin{align*}
\bm{\ell}\left([ (1/2)T, (5/2)T ];y\right) 
=& \left( \bigcup_{\substack{\rho = \beta + i \gamma;\\
\beta > (1/2)( 1/2 + \sigma_0 ),\\
\gamma \in [ (1/2)T, (5/2) T ]}} \left( \gamma - (y + 3),  \gamma + ( y +3 )\right)\right) \\
& \quad \cup [(1/2)T, (1/2)T + (y + 3) ] 
\cup [ (5/2)T - (y + 3), (5/2)T]. 
\end{align*}
Then we have
\[
\log\zeta( \sigma + i t ) 
= \sum_{2 \leq n \leq y} \frac{\Lambda(n)}{n^{\sigma + it} \log n } 
+ O( y^{(1/2 - \sigma_0)/2} (\log T)^3 )
\]
for $\sigma_0 \leq \sigma \leq 1$ and $t \in [ (1/2)T, (5/2) T ] \setminus \bm{\ell}\left([ (1/2)T, (5/2) T ];y\right)$, and the estimate
\[
\meas \left( \bm{\ell}\left([ (1/2)T, (5/2)T ];y\right) \right) \ll T^{5/4 - \sigma_0/2 } y (\log T)^5
\]
holds.
\end{lemma}

Fix $\sigma > 1/2$, and let $\sigma_0 = \frac{1}{2}\left(\frac{1}{2} + \sigma \right)$. 
Let $\eta = \left(\log\gamma(2T) \right)^{\frac{14}{\sigma_0 - \frac{1}{2}}}$.  
We define the set $\mathcal{I}(T)$ as 
\[
\mathcal{L}_\gamma(T) = \{t \in [T, 2T] : \gamma(t) \notin \bm{\ell}\left([\gamma(T/2), \gamma(5T/2) ];\eta\right)\}, 
\]
where $\bm{\ell}$ is defined in Lemma~\ref{appro} replacing $T/2, 5T/2$ by $\gamma(T/2)$, $\gamma(5T/2)$. 
Here, for a subset $I \subset [T, 2T]$, define 
\[
\mathbb{P}_{I}(f(t) \in A) := \frac{1}{T} \meas\{t \in I : f(t) \in A\}, 
\]
for every measurable function $f$ and measurable set $A$. 
Then, for any rectangle $\mathcal{R}$, we see that 
\begin{align*}
    &\mathbb{P}_T(\log\zeta(\sigma + i\gamma(t)) \in \mathcal{R}) \\
    &= \mathbb{P}_{\mathcal{J}_{\gamma}(T)}(\log\zeta(\sigma + i\gamma(t)) \in \mathcal{R}) 
    + O\left(\frac{1}{T}\meas\{[T, 2T] \setminus \mathcal{J}(T)\} \right). 
\end{align*}
Here, the implicit constant does not depend on $\gamma$ and $\mathcal{R}$. 

\begin{lemma} \label{measure error}
    Let $\gamma \in \mathcal{F}'$. 
    Then, for sufficiently small $\varepsilon > 0$, 
    \[
    \frac{1}{T}\meas\{[T, 2T] \setminus \mathcal{L}_{\gamma}(T)\} \ll_\varepsilon \gamma\left(\frac{T}{2} \right)^{\frac12(\frac{1}{2} - \sigma_0) + \varepsilon}. 
    \]
    Hence, we have $\mathcal{L}_\gamma(T) \sim T$. 
\end{lemma}

\begin{proof}
    By the definition of $\mathcal{F}'$, for the interval $(\gamma_\rho - (\eta+3), \gamma_\rho + (\eta + 3))$, the measure of the set $t$ satisfying $\gamma(t) \in (\gamma_\rho - (\eta+3), \gamma_\rho + (\eta + 3))$ is bounded by   
    \begin{align*}
     \gamma^{-1}(\gamma_\rho + (\eta + 3)) - \gamma^{-1}(\gamma_\rho - (\eta+3)) \ll \eta \frac{\gamma^{-1}(\gamma_\rho + \eta + 3)}{\gamma_\rho}. 
    \end{align*}
    Here, $\gamma_\rho$ is the imaginary part of $\rho = \beta_\rho + i\gamma_\rho$ with $\zeta(\rho) = 0$ and $\gamma(T/2) \le \gamma_\rho \le \gamma(5T/2)$. 
    Hence by the zero density estimate (cf. \cite[Theorem~9.19.A]{Ti}) and (F4), we obtain 
    \begin{align*}
        \frac{1}{T}\meas\{[T, 2T] \setminus \mathcal{L}_{\gamma}(T)\} &\ll \frac{\eta}{T}\sum_{\substack{\beta_\rho > \frac12(\frac12 + \sigma_0) \\ \gamma_\rho \in [\gamma(T/2), \gamma(5T/2)]}} \frac{\gamma^{-1}(\gamma_\rho + \eta + 3)}{\gamma_\rho} \\ 
        & \ll \eta\sum_{\substack{\beta_\rho > \frac12(\frac12 + \sigma_0) \\ \gamma_\rho \in [\gamma(T/2), \gamma(5T/2)]}} \frac{1}{\gamma_\rho} \\
        & \ll \eta\gamma\left(\frac{T}{2} \right)^{\frac{1}{2} - \frac12(\frac12 + \sigma_0)} \left(\log\gamma\left( \frac{5T}{2}\right)\right)^5 \ll_\varepsilon \gamma\left(\frac{T}{2} \right)^{\frac12(\frac{1}{2} - \sigma_0) + \varepsilon}.
    \end{align*}
\end{proof}

The key new ingredient is the use of condition (F4), which enables us to estimate the preimage of the exceptional set under the general shift. 
Lemma~\ref{measure error} implies 
\begin{align*}
    &\mathbb{P}_T(\log\zeta(\sigma + i\gamma(t)) \in \mathcal{R}) \\
    &= \mathbb{P}_{\mathcal{L}_{\gamma}(T)}(\log\zeta(\sigma + i\gamma(t)) \in \mathcal{R}) 
    + O\left(\gamma(T/2)^{-c} \right), 
\end{align*}
for a sufficiently small constant $c  =c(\sigma) > 0$. 

Next, we consider a subset of $\mathcal{L}_\gamma(T)$ such that $\log\zeta$ can be approximated by a Dirichlet polynomial with smaller length.
Let $Y = Y(T) = \log(\gamma(T))^{\frac{12}{2\sigma-1}}$. 

\begin{lemma} \label{new}
     Let $1/2 < \sigma < 1$ and $\gamma \in \mathcal{F}'$. 
    Then, as $T \to \infty$
    \begin{align*}
    &\frac{1}{T} \meas\{t \in \mathcal{L}_\gamma(T) : |P_\eta(\sigma + i\gamma(t)) - P_Y(\sigma + i\gamma(t))| > (\log \gamma(T))^{-2\sigma}  \} \ll (\log\gamma(T))^{-8}. 
    \end{align*}
    As consequently, there exists $\mathcal{I}_\gamma(T) \subset \mathcal{L}_\gamma(T) \subset [T, 2T]$ such that 
    \[
\log\zeta(\sigma + i\gamma(t)) = P_Y(\sigma + i\gamma(t)) + O(((\log \gamma(T))^{-2\sigma})
\]
holds for $t \in \mathcal{I}_\gamma(T)$ and 
\begin{align*}
     \frac{1}{T}\meas([T, 2T] \setminus \mathcal{I}_{\gamma}(T)) \ll (\log\gamma(T))^{-8}. 
\end{align*}
\end{lemma}

\begin{proof}
    By standard methods, we obtain 
    \begin{align*}
    \frac{1}{T}\int_{T}^{2T} \left|P_\eta(\sigma + i\gamma(t)) - P_Y(\sigma + i\gamma(t)) \right|^2 dt \ll Y^{1-2\sigma} + \frac{\eta^{2-2\sigma}}{T\tilde{\gamma}(T)}. 
    \end{align*}
    The condition (F4) implies $\eta^{2-2\sigma} = O(\gamma(T)^\varepsilon)$ for sufficiently small $0 < \varepsilon < 1$. 
    Therefore, we have 
    \begin{align*}
    \frac{1}{T}\int_{T}^{2T} \left|P_\eta(\sigma + i\gamma(t)) - P_Y(\sigma + i\gamma(t)) \right|^2 dt \ll (\log\gamma(T))^{-12}. 
    \end{align*}
    Applying the Chebyshev inequality, we evaluate 
    \begin{align*}
        &\frac{1}{T} \meas\{t \in \mathcal{L}_\gamma(T) : |P_\eta(\sigma + i\gamma(t)) - P_Y(\sigma + i\gamma(t))| > (\log \gamma(T))^{-2\sigma} \} \\
        &\ll (\log\gamma(T))^{4\sigma - 12} \ll (\log\gamma(T))^{-8}. 
    \end{align*}
    Therefore, there exists $\mathcal{I}_\gamma(T) \subset \mathcal{L}_\gamma(T)$ such that 
    \[
\log\zeta(\sigma + i\gamma(t)) = P_Y(\sigma + i\gamma(t)) + O(((\log \gamma(T))^{-2\sigma})
\]
holds for $t \in \mathcal{I}_\gamma(T)$. 
\end{proof}

By the above lemma, 
\begin{align*}
      &\mathbb{P}_T(\log\zeta(\sigma + i\gamma(t)) \in \mathcal{R}) \\
    &= \mathbb{P}_{\mathcal{I}_{\gamma}(T)}(\log\zeta(\sigma + i\gamma(t)) \in \mathcal{R}) 
    + O\left((\log\gamma(T))^{-4} \right)
\end{align*}
holds. 
Under this setting, the following lemma holds. 
\begin{lemma} \label{higher general}
    Let $1/2 < \sigma < 1$, and $\gamma \in \mathcal{F}'$.  
    Let $T$ be sufficiently large. 
    Then, for $k \ge 1$ satisfying \eqref{Y}, 
    \begin{align*}
        &\frac{1}{T}\int_{\mathcal{I}_{\gamma}(T)} |\log\zeta(\sigma + i\gamma(t))|^{2k}dt \\
        &\ll A^k\frac{k^{2k(1-\sigma_0)}}{(\log2k)^{2k\sigma_0}} + A^k(\log \gamma(T))^{-2k}, 
    \end{align*}
    where $A$ is a positive constant. 
    The implicit constant is independent of $k$. 
\end{lemma}

\begin{proof}
    By Lemma~\ref{appro}, we have 
    \begin{align*}
        &\frac{1}{T}\int_{\mathcal{I}_{\gamma}(T)} |\log\zeta(\sigma + i\gamma(t))|^{2k}dt \\
        &\le  \frac{2^k}{T} \int_{\mathcal{I}_{\gamma}(T)} \left| P_Y(\sigma + i\gamma(t))\right|^{2k}dt + \frac{2^k}{T} \int_{\mathcal{I}_{\gamma}(T)} \left|\log\zeta(\sigma + i\gamma(t)) - P_Y(\sigma + i\gamma(t))\right|^{2k}dt \\
        &=: S_1 + S_2. 
    \end{align*}
   By Lemma~\ref{higher P general}, the first term satisfies
   \begin{align*}
       S_1 \le \frac{2^k}{T} \int_{T}^{2T} \left| P_Y(\sigma + i\gamma(t))\right|^{2k}dt \le A^k\frac{k^{2k(1-\sigma_0)}}{(\log2k)^{2k\sigma_0}}. 
   \end{align*}
   Next, we estimate $S_2$. 
   Applying Lemma~\ref{new}, we see that 
   \begin{align*} 
   S_2 \ll 2^k (\log\gamma(T))^{-4k\sigma} \frac{1}{T}\meas(\mathcal{I}_\gamma(T)) \ll A^k (\log \gamma(T))^{-2k}. 
   \end{align*}
   This completes the proof. 
   
\end{proof}

Let $\gamma \in \mathcal{F}'$. 
Let $\mathcal{S}$ be the set of all rectangles in $\mathbb{C}$ whose sides are parallel to the coordinate axes.
Put $M = \log\log\gamma(T)$. 

Define $\mathcal{S}_M \subset \mathcal{S}$ as the set of all $\mathcal{R}$ such that  
\[
\mathcal{R} \subset \mathcal{R}_M := [-M, M]\times i[-M, M]. 
\]
For $\mathcal{R} \notin \mathcal{S}_{M}$, 
\begin{align*}
    &\mathbb{P}_{\mathcal{I}_{\gamma}(T)}(\log \zeta(\sigma + i\gamma(t)) \in \mathcal{R}) \\
    &= \mathbb{P}_{\mathcal{I}_{\gamma}(T)}(\log \zeta(\sigma + i\gamma(t)) \in \mathcal{R} \cap \mathcal{R}_M) + \mathbb{P}_{\mathcal{I}_{\gamma}(T)}(\log \zeta(\sigma + i\gamma(t)) \in \mathcal{R} \setminus \mathcal{R}_M)
\end{align*}
holds. 

By Lemma~\ref{higher general} with $k = \lfloor M \rfloor$, we deduce 
\begin{align*}
    \mathbb{P}_{\mathcal{I}_{\gamma}(T)}(\log \zeta(\sigma + i\gamma(t)) \in \mathcal{R} \setminus \mathcal{R}_M) 
    &\ll \sum_{\alpha \in \{0, \frac{\pi}{2}, \pi, \frac{3\pi}{2}\}}\mathbb{P}_{\mathcal{I}_{\gamma}(T)} (\Re e^{-i\alpha}\log\zeta(\sigma + i\gamma(t)) > M) \\
    &\ll (\log\gamma(T))^{-2}. 
\end{align*}
Hence, 
\[
\mathbb{P}_{\mathcal{I}_{\gamma}(T)} (\log\zeta(\sigma + i\gamma(t)) \in \mathcal{R}) = \mathbb{P}_{\mathcal{I}_{\gamma}(T)} (\log\zeta(\sigma + i\gamma(t)) \in \mathcal{R} \cap \mathcal{R}_{M}) + O\left((\log\gamma(2T))^{-2}\right)
\]
holds. 
Applying \cite[Lemma~5.2]{EIM}, we can also show that 
\[
\mathbb{P}(\log\zeta(\sigma, X) \in \mathcal{R}) = \mathbb{P}(\log\zeta(\sigma, X) \in \mathcal{R} \cap \mathcal{R}_M) + O\left((\log\gamma(2T))^{-2} \right).
\]
Combining these arguments, we obtain 
\begin{align*}
    &D_{\sigma, \gamma}(T) = \sup_{\mathcal{R} \in \mathcal{S}} \left|\mathbb{P}_T(\log\zeta(\sigma + i\gamma(t)) \in \mathcal{R}) - \mathbb{P}(\log\zeta(\sigma, X) \in \mathcal{R})\right| \\
    &= \sup_{\mathcal{R} \in \mathcal{S}} \left|\mathbb{P}_{\mathcal{I}_{\gamma}(T)}(\log\zeta(\sigma + i\gamma(t)) \in \mathcal{R}) - \mathbb{P}(\log\zeta(\sigma, X) \in \mathcal{R})\right| + O\left((\log\gamma(T))^{-8} \right) \\
    &= \sup_{\mathcal{R} \in \mathcal{S}} \left|\mathbb{P}_{\mathcal{I}_{\gamma}(T)}(\log\zeta(\sigma + i\gamma(t)) \in \mathcal{R} \cap \mathcal{R}_M) - \mathbb{P}(\log\zeta(\sigma, X) \in \mathcal{R} \cap \mathcal{R}_M)\right|\\
    & \quad + O\left((\log\gamma(2T))^{-2} + (\log\gamma(T))^{-8} \right) \\
    &= \sup_{\mathcal{R} \in \mathcal{S}_M} \left|\mathbb{P}_{\mathcal{I}_{\gamma}(T)}(\log\zeta(\sigma + i\gamma(t)) \in \mathcal{R}) - \mathbb{P}(\log\zeta(\sigma, X) \in \mathcal{R} )\right| + O\left((\log\gamma(T))^{-2} \right) .
\end{align*}

In addition, we define 
\[
D_{\sigma, \gamma, Y}(T) = \sup_{\mathcal{R} \in \mathcal{S}_{2M}} \left|\mathbb{P}_{\mathcal{I}_{\gamma}(T)}(P_Y(\sigma + i\gamma(t)) \in \mathcal{R}) - \mathbb{P}(\log\zeta(\sigma, X) \in \mathcal{R} )\right|. 
\]
Putting $\varepsilon = (\log\gamma(T))^{-2\sigma}$, we have 
\begin{equation} \label{easy}
\log\zeta(\sigma + i\gamma(t)) = P_Y(\sigma + i\gamma(t)) + O(\varepsilon)
\end{equation}
for $t \in \mathcal{I}_\gamma(T)$. 
For $\mathcal{R} = (a_1, b_1) \times i(a_2, b_2) \in \mathcal{S}_M$, denote $R^{+\varepsilon} = (a_1-\varepsilon, b_1 +\varepsilon) \times i(a_2 - \varepsilon, b_2 + \varepsilon)$, and $R^{-\varepsilon} = (a_1+\varepsilon, b_1 -\varepsilon) \times i(a_2 + \varepsilon, b_2 - \varepsilon)$. 
Then, with \eqref{easy} one can obtain
\begin{equation*}
    \mathbb{P}_{\mathcal{I}_{\gamma}(T)}(P_Y(\sigma + i\gamma(t)) \in \mathcal{R}^{-\varepsilon}) \le 
    \mathbb{P}_{\mathcal{I}_{\gamma}(T)}(\log\zeta(\sigma + i\gamma(t)) \in \mathcal{R})
    \le \mathbb{P}_{\mathcal{I}_{\gamma}(T)}(P_Y(\sigma + i\gamma(t)) \in \mathcal{R}^{+\varepsilon}). 
\end{equation*}
If $T$ is sufficiently large, then $\mathcal{R}^{\pm\varepsilon} \in \mathcal{S}_{2M}$. 
Therefore, by the definition of $D_{\sigma, \gamma, Y}$, and \cite[Proposition~3.2]{EIM}, for $\mathcal{R} \in \mathcal{S}_M$, we have
\begin{align*}
|\mathbb{P}_{\mathcal{I}_{\gamma}(T)}(\log\zeta(\sigma + i\gamma(t)) \in \mathcal{R}) - \mathbb{P}(\log\zeta(\sigma, X) \in \mathcal{R})| \ll D_{\sigma, \gamma, Y}(T) + \varepsilon M. 
\end{align*}
Hence, $D_{\sigma, \gamma}(T)$ is bounded by
\[
D_{\sigma, \gamma}(T) \ll D_{\sigma, \gamma, Y}(T) + \varepsilon M + (\log\gamma(T))^{-2} . 
\]

To estimate $D_{\sigma, \gamma, Y}$, we introduce the Beurling--Selberg function (cf. \cite{EIM} and \cite{LLR}). 
Let
\begin{align*}
G(u) =  \frac{2u}{\pi} + \frac{2(1-u)u}{\tan(\pi u)}, \quad  K(x) = \left(\frac{\sin\pi x}{\pi x} \right)^2, \quad f_{a, b}(u) = \frac{e^{-2\pi i u a } - e^{-2\pi i u b}}{2}
\end{align*}
for $a, b \in \mathbb{R}$.  
Let $\bm{1}_A$ denote the indicator function of a set $A \subset \mathbb{R}$. 
Then, the indicator of a set $\mathcal{R} \in \mathcal{S}$ is represented as $\bm{1}_\mathcal{R}(z) = \bm{1}_{(a_1, b_1)}(u)\bm{1}_{(a_2, b_2)}(v)$ for $z = u + iv$. 

\begin{lemma}[{\cite[Section~5]{EIM}}] \label{beuring selberg}
    Let $\mathcal{R} = (a_1, b_1) \times i(a_2, b_2) \in \mathcal{S}$, and $L > 0$. 
    Then, 
    \begin{align*} \label{1 appro}
        \bm{1}_{\mathcal{R}}(z) &= W_{L, \mathcal{R}}(z) \\ &+ O(K(L(x-a_1)) + K(L(x-b_1)) + K(L(x-a_2)) + K(L(x-b_2))), 
    \end{align*}
    where
    \begin{align*}
        W_{L, \mathcal{R}}(z) = \frac12 &\Re\int_0^L\int_0^L G\left(\frac{u}{L} \right)G\left(\frac{v}{L} \right) \\ 
        &\times \left(e^{2\pi i(ux-vy)}f_{a_1, b_1}(u)\overline{f_{a_2, b_2}(v)} - e^{2\pi i(ux +vy)}f_{a_1, b_1}(u)f_{a_2, b_2}(v) \right) \frac{du}{u}\frac{dv}{v}.
    \end{align*}
\end{lemma}

For $W_{L, \mathcal{R}}$, we deduce the following proposition. 

\begin{proposition} \label{W general}
    Let $\gamma \in \mathcal{F}'$, $1/2 < \sigma < 1$ and $L = c_3(\log\gamma(T))^{\sigma}$, where $c_3$ is a small positive constant. 
    Then, 
    \[
    \frac{1}{T} \int_T^{2T} W_{L, \mathcal{R}}(P_Y(\sigma + i\gamma(t)))dt = \mathbb{E}[W_{L, \mathcal{R}}(\log\zeta(\sigma, X))] + O\left((\log\gamma(T))^{-2} \right)
    \]
    for $\mathcal{R} \in \mathcal{S}_{2M}$, and $w = u + iv$. 
    
\end{proposition}

\begin{proof}
    Define $\langle z, w \rangle := \Re(z)\Re(w) + \Im(z)\Im(w) = \frac{1}{2}(z\overline{w} + \overline{z}w)$, and 
    \begin{align*}
        \Lambda_{\sigma, \gamma, T}(w) &= \frac{1}{T}\int_{T}^{2T} \exp\left(i\langle P_Y(\sigma + i\gamma(t)), w \rangle \right)dt \\
        &= \frac{1}{T} \int_T^{2T}\exp \left(\frac{i\overline{w}}{2}P_Y(\sigma + i\gamma(t)) + \frac{iw}{2} \overline{P_Y(\sigma + i\gamma(t))}\right)dt. 
    \end{align*}
By the definition of $W_{L, \mathcal{R}}$, the left-hand side of the assertion can be written as
\begin{align*}
     &\frac{1}{T} \int_T^{2T} W_{L, \mathcal{R}}(P_Y(\sigma + i\gamma(t)))dt \\ 
     &= \frac12 \Re\int_0^L\int_0^L G\left(\frac{u}{L} \right)G\left(\frac{v}{L} \right) \\
        &\times \left(\Lambda_{\sigma, \gamma, T}(2\pi \overline{w}) f_{a_1, b_1}(u)\overline{f_{a_2, b_2}(v)} - \Lambda_{\sigma, \gamma, T}(2\pi w)f_{a_1, b_1}(u)f_{a_2, b_2}(v) \right) \frac{du}{u}\frac{dv}{v}.
\end{align*}
Let $V = c_4 (\log\gamma(T))^{1-\sigma}(\log\log\gamma(T))^{-1}$ for a sufficiently small constant $c_4 = c_4(\sigma) \le 1/2$. 
By Lemma~\ref{large deviation general}, we derive 
\begin{align*}
    \Lambda_{\sigma, \gamma, T}(2\pi w) &= \frac{1}{T}\int_{A_T}\exp\left(i\pi \overline{w}P_Y(\sigma + i\gamma(t)) + i \pi w \overline{P_Y(\sigma + i\gamma(t))}  \right)dt \\
    &+O\left(\exp\left(-c_1 V^{\frac{1}{1-\sigma}}(\log V)^{\frac{\sigma}{1-\sigma}}\right) \right),  
\end{align*}
and we have 
\[
|i\pi w| \le d_1(V\log V)^{\frac{\sigma}{1-\sigma}}, 
\]
where $d_1$ is given by Proposition~\ref{prop1 general}.

Define the characteristic function of $\log\zeta(\sigma, X)$ as
\[
\Lambda_{\sigma}(w) = \mathbb{E}[\exp(i\langle \log\zeta(\sigma, X), w \rangle)]
\]
for $w \in \mathbb{C}$. 
From Proposition~\ref{prop1 general} and \cite[Lemma~3.4]{EIM}, for $|u|, |v| \le L$, it can be deduced that 
\begin{equation} \label{Lambda} 
\Lambda_{\sigma, \gamma, T}(2\pi w) = \Lambda_{\sigma}(2\pi w) + E_1, 
\end{equation}
where 
\begin{align*}
    E_1 &\ll \exp\left(-c_1 V^{\frac{1}{1-\sigma}}(\log V)^{\frac{\sigma}{1-\sigma}}\right) + \frac{1}{T\tilde{\gamma}(T)}\left( (V\log V)^{\frac{\sigma}{1-\sigma}} Y\right)^{V^\frac{1}{1-\sigma}(\log V)^{\frac{\sigma}{1-\sigma}}} \\
    &+\exp\left(-d_2 V^{\frac{1}{1-\sigma}}(\log V)^{\frac{\sigma}{1-\sigma}}\right) + \frac{L}{Y^{\sigma - \frac12}} \\
    &\ll (\log \gamma(T))^{-5}. 
\end{align*}

Substituting \eqref{Lambda} into the above integral, 
\begin{align*}
     &\frac{1}{T} \int_T^{2T} W_{L, \mathcal{R}}(P_Y(\sigma + i\gamma(t)))dt \\
     &=\frac12 \Re\int_0^L\int_0^L G\left(\frac{u}{L} \right)G\left(\frac{v}{L} \right) \\
        &\times \left(\Lambda_{\sigma}(2\pi \overline{w}) f_{a_1, b_1}(u)\overline{f_{a_2, b_2}(v)} - \Lambda_{\sigma}(2\pi w)f_{a_1, b_1}(u)f_{a_2, b_2}(v) \right) \frac{du}{u}\frac{dv}{v} + E_2 \\
    &= \mathbb{E}[W_{L, \mathcal{R}}(\log\zeta(\sigma, X))] + E_2, 
\end{align*}
where 
\[
E_2 \ll L^2(\log\gamma(T))^{-5}(b_1 - a_1)(b_2 - a_2) \ll (\log\gamma(T))^{-2}. 
\]
Thus, we get the conclusion. 
\end{proof}

\begin{proof}[Proof of Theorem~\ref{main general}]
Let $\gamma \in \mathcal{F}'$. 
Let $\mathcal{R} = (a_1, b_1) \times i(a_2, b_2) \in \mathcal{S}_{2M}$. 
Recall that
\[
D_{\sigma, \gamma}(T) \ll D_{\sigma, \gamma, Y}(T) + (\log\gamma(T))^{-2\sigma}\log\log\gamma(T) + (\log\gamma(T))^{-2}. 
\]
Using Proposition~\ref{W general}, we see that
\begin{align*}
    \mathbb{P}_{\mathcal{I}_{\gamma}(T)}(P_Y(\sigma + i\gamma(t)) \in \mathcal{R}) 
    &= \mathbb{P}_{T}(P_Y(\sigma + i\gamma(t)) \in \mathcal{R}) +O((\log \gamma(T))^{-8})\\
    &= \frac{1}{T}\int_{T}^{2T}W_{L, \mathcal{R}}(P_Y(\sigma + i\gamma(t)))dt + +O((\log\gamma(T))^{-8}) \\
    &+O(I_T(a_1) + I_T(b_1) + J_T(a_2) + J_T(b_2)), 
\end{align*}
where $I_T$, and $J_T$ are given by 
\begin{align*}
    I_T(\xi) = \frac{1}{T}\int_T^{2T}K(L(\Re P_Y(\sigma + i\gamma(t)) - \xi))dt, \\
    J_T(\xi) = \frac{1}{T}\int_T^{2T}K(L(\Im P_Y(\sigma + i\gamma(t)) - \xi))dt. 
\end{align*}
By the fact 
\[
K(Lx) = \frac{2}{L^2} \Re \int_0^{L} (L-u)e^{2\pi i xu}du 
\]
(cf. \cite[(6.10)]{LLR}), the same argument as in \cite[Section~5]{EIM} and \eqref{Lambda}, we have 
\begin{align*}
    I_T(\xi) = \frac{2}{L^2} \Re\int_{0}^L(L-u)e^{-2\pi i \xi u} \Lambda_{\sigma}(2\pi u)du + O\left((\log\gamma(T))^{-4} \right), \\
    J_T(\xi) = \frac{2}{L^2} \Re\int_{0}^L(L-v)e^{-2\pi i \xi v} \Lambda_{\sigma}(2\pi iv)dv + O\left((\log\gamma(T))^{-4} \right). 
\end{align*}
Moreover, from \cite[(3.5)]{EIM}, $I_T$ and $J_T$ are evaluated as  
\[
I_T(\xi), J_T(\xi) \ll L^{-1} + (\log \gamma(T))^{-4}. 
\]
Therefore, we obtain 
\[
\mathbb{P}_T(P_Y(\sigma + i\gamma(t)) \in \mathcal{R}) = \frac{1}{T}\int_T^{2T}W_{L, \mathcal{R}}(P_Y(\sigma +i\gamma(t))dt + O\left(L^{-1} + (\log\gamma(T))^{-4} \right). 
\]
Similarly, we deduce that 
\[
\mathbb{P}(\log\zeta(\sigma, X) \in \mathcal{R}) = \mathbb{E}[W_{L, \mathcal{R}}(\log\zeta(\sigma, X))] + O\left(L^{-1} + (\log\gamma(T))^{-4} \right). 
\]
By Proposition~\ref{W general}, we have
\begin{align*}
    &\mathbb{P}_T(P_Y(\sigma + i\gamma(t)) \in \mathcal{R}) - \mathbb{P}(\log\zeta(\sigma, X) \in \mathcal{R}) \\
    &=\frac{1}{T}\int_T^{2T}W_{L, \mathcal{R}}(P_Y(\sigma +i\gamma(t))dt - \mathbb{E}[W_{L, \mathcal{R}}(\log\zeta(\sigma, X))] + O\left(L^{-1} + (\log\gamma(T))^{-4} \right) \\ 
    &\ll (\log\gamma(T))^{-2} + L^{-1} \ll (\log\gamma(T))^{-\sigma}. 
\end{align*}
Therefore, by the definition of $D_{\sigma, \gamma, Y}(T)$, the above evaluation yields the conclusion. 
\end{proof}

\section{Preliminaries for Theorem~\ref{main1}}
Throughout the paper, we assume that $1<\alpha_1<\cdots<\alpha_r$. 
Define 
\[
\log\underline{\zeta}(s + i\underline{\log}\tau) = (\log\zeta(s + i(\log \tau)^{\alpha_1}), \dots, \log\zeta(s + i(\log\tau)^{\alpha_r})).
\] 

Let $\mathcal K$ be a compact subset satisfying the assumptions of Theorem~\ref{main1}. 
We put $\sigma' = (x_\Phi(\alpha_1) + \min_{s \in \mathcal{K}} \Re(s))/2$, $\sigma_1 = (\sigma' + \min_{s \in \mathcal{K}} \Re(s))/2$ and $\sigma_2 = (1 + \max_{s \in \mathcal{K}} \Re(s))/2$.
Then, we define the rectangular region $\mathcal{R}$ by 
\begin{equation} \label{definition of R}
\mathcal{R} = (\sigma_1,\ \sigma_2) \times i \left( \min_{s \in \mathcal{K}} \Im(s) - 1/2,\ \max_{s \in \mathcal{K}}\Im(s) + 1/2 \right).
\end{equation}

There exists a sequence of compact subsets $K_\ell$ of $\mathcal{R}$, $\ell=1, 2, \dots$ satisfying the following properties: 
\begin{itemize} 
 \item $\mathcal{R} = \bigcup_{\ell=1}^{\infty} K_\ell$, 
 \item $K_\ell \subset K_{\ell+1}$ for any $\ell \in \mathbb{N}$,
 \item For all compact subset $K$ of $\mathcal{R}$, there exists $\ell \in \mathbb{N}$ such that $K \subset K_\ell$ 
\end{itemize}
(see \cite[Chapter VII, 1.2 Proposition]{Co}). 
Let $\mathcal{H}(\mathcal{R})$ be the set of all holomorphic functions on $\mathcal{R}$. 
For $g_1, g_2 \in \mathcal{H}(\mathcal{R})$, let $d_\ell(g_1, g_2) = \sup_{s \in K_\ell} |g_1(s) - g_2(s) |$ and put  
\[
d(g_1, g_2) = \sum_{\ell=1}^{\infty} \frac{1}{2^\ell} \frac{d_\ell(g_1, g_2)}{1 + d_\ell(g_1, g_2)}.
\]
Then $d$ defines a metric on $\mathcal H(\mathcal R)$. 
Moreover, the induced topology is the topology of compact convergence. 
Similarly, define $\underline{d}(\underline{g}_1, \underline{g}_2) = \max_{1 \le j \le r} d(g_{1j}, g_{2j})$, 
where $\underline{g}_k = (g_{k1}, \dots, g_{kr}) \in \mathcal{H}^r(\mathcal{R})$ for $k = 1, 2$. 
Then $\underline{d}$ is a metric on $\mathcal{H}^r(\mathcal{R})$. 

Therefore, $\log\underline\zeta(s,X)$ defines an $\mathcal H^r(\mathcal R)$-valued random element. 
We define probability measures on $(\mathcal{H}^r(\mathcal{R}), \mathcal{B}(\mathcal{H}^r(\mathcal{R})))$ by 
\begin{equation*}
\mathcal{Q}_T(A) = \frac{1}{T} \mathrm{meas} \left\{\tau \in [T, 2T] : \log\underline{\zeta}(s + i\underline{\log}(\tau)) \in A \right\}, 
\end{equation*}
\begin{equation*}
\mathcal{Q}(A) = \mathbf{m}^r \left\{X \in \Omega^r : \log\underline{\zeta}(s, X) \in A \right\}
\end{equation*} 
for $A \in \mathcal{B}(\mathcal{H}^r(\mathcal{R}))$.

\section{A limit theorem for Theorem~\ref{main1}}

The goal of this section is to prove the following proposition.

\begin{proposition} \label{prop uni}
The probability measure $\mathcal{Q}_T$ converges weakly to $\mathcal{Q}$ as $T \to \infty$.
\end{proposition}

First, we introduce the following notation.
\begin{itemize}
 \item $|\mathcal{K}| = \max_{s \in \mathcal{K}} \Im(s) - \min_{s \in \mathcal{K}} \Im(s)$.
 \item $\tau_0 = \tau_0(\mathcal{K}) = (\max_{s \in \mathcal{K}}\Im(s) + \min_{s \in \mathcal{K}} \Im(s))/2$. 
 \item $\sigma_0 =\sigma_0(\mathcal{K}) = (x_\Phi(\alpha_1) + \min_{s \in \mathcal{K}} \Re(s))/2$.
 \item For any $\Delta >0$, let
 \begin{equation} \label{exceptional}
 \mathcal{G}_{\sigma_0, \Delta} = \mathbb{R} \setminus \left\{ \left( \bigcup_{\substack{\rho= 
  \beta + i\gamma \\ \beta > \sigma_0}} (\gamma 
  -\tau_0(\mathcal{K})- \Delta, \gamma - \tau_0(\mathcal{K}) + \Delta) \right) 
  \cup   (-\tau_0(\mathcal{K}) - \Delta, -\tau_0(\mathcal{K}) + \Delta ) \right\}.
 \end{equation}

 \item For any $T >0$, let $\mathcal{I}_\mathcal{K}(T) = \mathcal{G}_{\sigma_0, |\mathcal{K}| +1} \cap [T, 2T]$. 
\end{itemize}
 Note that $\mathcal{I}_{\mathcal{K}}(T) \sim T$ as $T \to \infty$.

\begin{lemma} \label{meas log}
    Let $\alpha \ge \alpha_1$. Assume $\Delta = O_\varepsilon((\log T)^\varepsilon)$. 
   Then  
    \[
    \frac{1}{T}\meas\{t \in [T, 2T] : (\log t)^\alpha \notin \mathcal{G}_{\sigma_0, \Delta} \} \ll \Delta (\log T)^{1 -\alpha + \alpha\Phi(\sigma_0) + \varepsilon}     \]
    holds for sufficiently large $T > 0$. 
    In particular, 
    \[
    \meas\{t \in [T, 2T] : (\log t)^\alpha \in \mathcal{G}_{\sigma_0, \Delta} \} \sim T
    \]
    holds for $\sigma_0 > x_\Phi(\alpha_1)$. 
\end{lemma}

\begin{proof}
    First, we estimate 
    \[
    \meas\{t : (\log t)^\alpha \in (\gamma - \tau_0 - \Delta, \gamma - \tau_0 + \Delta)\}. 
    \]
    This measure satisfies
    \begin{align*}
        \meas\{t \in (\exp((\gamma - \tau_0 - \Delta)^{\frac{1}{\alpha}}), \exp((\gamma - \tau_0 + \Delta)^{\frac{1}{\alpha}})  \} \ll \frac{\Delta}{\alpha}(\gamma - \tau_0 + c_\gamma)^{\frac{1}{\alpha}-1} \exp\left((\gamma -\tau_0 + c_\gamma)^{\frac{1}{\alpha}}  \right),
    \end{align*}
    where $c_\gamma$ is a constant with $-\Delta < c_\gamma < \Delta$. 
    Hence, we have
    \begin{align*}
        &\frac{1}{T}\meas\{t \in [T, 2T] : (\log t)^\alpha \notin \mathcal{G}_{\sigma_0, \Delta} \}  \\
        &\ll \frac{\Delta}{\alpha} \sum_{\substack{\gamma - \tau_0 \in [(\log T)^\alpha - \Delta, (\log2T)^\alpha + \Delta] \\ \beta > \sigma_0}} (\gamma - \tau_0 + c_\gamma)^{\frac{1}{\alpha}-1} \exp\left((\gamma -\tau_0 + c_\gamma)^{\frac{1}{\alpha}}  \right) \\ 
        &\ll \frac{\Delta T}{\alpha(\log T)^{\alpha-1}} N(\sigma_0, (\log 2T)^\alpha + \Delta) 
        \ll \Delta (\log T)^{1 -\alpha + \alpha\Phi(\sigma_0) + \varepsilon} . 
    \end{align*}
    
Since we assume $\sigma_0 > x_\Phi(\alpha_1)$, 
\[
\frac{1}{T}\meas\{t \in [T, 2T] : (\log t)^\alpha \notin \mathcal{G}_{\sigma_0, \Delta} \}  \to 0
\]
as $T \to \infty$. 
Therefore, we get $\meas\{t \in [T, 2T] : (\log t)^\alpha \in \mathcal{G}_{\sigma_0, \Delta} \} \sim T$ as $T \to \infty$.
\end{proof}

We prove the mean value estimates for logarithmic shifts. 

\begin{lemma} \label{mean1}
    Let $\alpha > 1$. 
Let $\{a_n\}_{n \ge 1}$ be complex numbers that satisfy $a_n \ll_{\varepsilon} n^{\varepsilon}$ for any $\varepsilon > 0$. 
Then, 
\[
\int_{T}^{cT} \left| \sum_{n \le Y} \frac{a_n}{n^{\sigma + i(\log t)^\alpha}} \right|^2 \, dt 
\ll_{c, \sigma, \varepsilon} T\left(\sum_{n \le Y}\frac{|a_n|^2}{n^{2\sigma}} +  (\log T)^{1- \alpha}Y^{2-2\sigma + \varepsilon} \right)
\]
for sufficiently large $T > 0$, $c > 1$, $1/2 < \sigma < 1$ and $0 < \varepsilon < \sigma -1/2$, and uniformly for $1/2 < \sigma' \le \sigma \le \sigma'' < 1$.
\end{lemma}

\begin{proof}
    By simple calculations, we have
    \begin{align*}
        &\int_{T}^{cT} \left| \sum_{n \le Y} \frac{a_n}{n^{\sigma + i(\log t)^\alpha}} \right|^2 \, dt \\
&= \int_{T}^{cT} \sum_{m \le Y} \frac{a_m}{m^{\sigma + (\log t)^\alpha}} \sum_{n \le Y} \frac{\overline{a_n}}{n^{\sigma - i(\log t)^\alpha}} \, dt \\
&= \sum_{m, n \le Y} \frac{a_m \overline{a_n}}{(mn)^\sigma} \int_{T}^{cT} \left(\frac{n}{m} \right)^{i(\log t)^\alpha} \, dt  \\ 
&= (c-1)T\sum_{n \le Y}\frac{|a_n|^2}{n^{2\sigma}}  + \sum_{\substack{m, n \le Y \\ m \ne n}}\frac{a_m \overline{a_n}}{(mn)^\sigma} \int_{T}^{cT} \left(\frac{n}{m} \right)^{i(\log t)^\alpha} \, dt.  
    \end{align*}

We estimate the second term. 
By the first derivative test, we deduce
\begin{align*}
\int_{T}^{cT} \left(\frac{n}{m} \right)^{i(\log t)^\alpha} \, dt 
\ll \frac{T}{\alpha(\log cT)^{\alpha-1} \left|\log\frac{n}{m}\right|}. 
\end{align*}
Thus, the second term is bounded by 
\begin{align*}
    \sum_{\substack{m, n \le Y \\ m \ne n}}\frac{a_m \overline{a_n}}{(mn)^\sigma} \int_{T}^{cT} \left(\frac{n}{m} \right)^{i(\log t)^\alpha}dt &\ll \frac{T}{\alpha(\log cT)^{\alpha-1}} \sum_{m < n \le Y} \frac{|a(m)||a(n)|}{(mn)^\sigma \log\frac{n}{m}} \\ 
    &\ll_{\varepsilon, \alpha, c} T (\log T)^{1- \alpha}Y^{2-2\sigma + \varepsilon}. 
\end{align*}

\end{proof}

\begin{lemma} \label{lemma smooth uni}
Let $\varphi : [0, \infty) \to \mathbb{C}$ be a smooth function such that $\varphi$ and all its derivatives decay faster than any polynomial at infinity, and let $\hat{\varphi}(s) = \int_{0}^{\infty} \varphi(x) x^{s-1}\, dx$ be the Mellin transform of $\varphi$ on $\Re(s) > 0$.
\begin{enumerate}
\item[(1)] The Mellin transform $\hat{\varphi}$ extends to a meromorphic function on $\Re(s) > -1$, with at most a simple pole at $s = 0$ with residue $\varphi(0)$.
\item[(2)] For any real numbers $-1 < A < B$, the Mellin transform has rapid decay in the strip $A \le \sigma \le B$, in the sense that for any integer $k \ge 1$, there exists a constant $C = C(k, A, B) \ge 0$ such that 
\[
|\hat{\varphi}(\sigma + it) | \le C (1 + |t|)^{-k}.
\]
for all $A \le \sigma \le B$ and $|t| \ge 1$.
\item[(3)] For any $\sigma > 0$ and any $x \ge 0$, we have the Mellin inversion formula 
\[
\varphi(x) = \frac{1}{2 \pi i}\int_{\sigma - i \infty}^{\sigma + i \infty} \hat{\varphi}(s)x^{-s}\, ds.
\]

\end{enumerate}
\end{lemma}

\begin{proof}
    The proof can be found in \cite[Appendix~A]{Ko}.
\end{proof}

Fix a real-valued smooth function $\varphi(x)$ on $[0, \infty)$ with compact support satisfying $\varphi(x) = 1$ for $0 \le x \le 1$ and $0 \le \varphi(x) \le 1$ for $x \ge 0$. 
We put 
\[
\log\zeta_X(\sigma + it) = \sum_{n=2}^{\infty} \frac{\Lambda(n)\varphi(n/X)}{n^{\sigma + it}\log n}, \quad \log\zeta_X(s, X) = \sum_{n=2}^\infty \frac{\Lambda(n)X(n)\varphi(n/X)}{n^{s}\log n}
\]
for $X \ge 2$. 
We next prove that $\log \zeta$ can be approximated by $\log\zeta_X$. 

Let $Y = Y(T) = (\log\log T)^{\frac{8}{\sigma_0 - \frac12}}$, $\mathcal{X}_{\mathcal{K}, j}(T) = \{t \in [T, 2T] : (\log t)^{\alpha_j} \in \mathcal{G}_{\sigma_0, |\mathcal{K}| + Y + 4} \}$ and $\mathcal{X}_{\mathcal{K}}(T) = \bigcap_{j=1}^{r} \mathcal{X}_{\mathcal{K}, j}(T)$. 
By Lemma~\ref{meas log}, we see that $\meas([T, 2T] \setminus \mathcal{X}_{\mathcal{K}, j}(T)) = o(T)$ as $T \to \infty$ and 
\[
\meas([T, 2T] \setminus \mathcal{X}_{\mathcal{K}}(T)) \le \sum_{j=1}^r \meas([T, 2T] \setminus \mathcal{X}_{\mathcal{K}, j}(T)) = o(T) 
\]
as $T \to \infty$. 
Therefore, we have $\meas(\mathcal{X}_\mathcal{K}(T)) \sim T$.  

\begin{lemma} \label{appro X}
Let $C$ be a compact set in $\mathcal{R}$. Then, for $j=1, \dots, r$,
    \[
    \lim_{X \to \infty} \limsup_{T \to \infty} \frac{1}{T} \int_{\mathcal{X}_{\mathcal{K}}(T)} \sup_{s \in C}\left|\log\zeta(s + i(\log \tau)^{\alpha_j}) - \log\zeta_X(s + i(\log\tau)^{\alpha_j}) \right| d\tau = 0. 
    \]
\end{lemma}

\begin{proof}

Let $T$ be sufficiently large.
For $\tau \in \mathcal{X}_\mathcal{K}(T)$,  $\log\zeta(s + i(\log\tau)^{\alpha_j})$ is holomorphic on $\overline{\mathcal{R}}$.
By Cauchy's integral formula, we have 
\[
\log\zeta(s +i(\log\tau)^{\alpha_j}) - \log\zeta_X(s + i(\log\tau)^{\alpha_j})
=\frac{1}{2\pi i}\int_{\partial{\mathcal{R}}} \frac{\log\zeta(z + i(\log\tau)^{\alpha_j}) - \log\zeta_X(z + i(\log\tau)^{\alpha_j})}{z - s}\, dz
\]
for any $s \in C$ and any $\tau \in \mathcal{X}_\mathcal{K}(T)$. 
Hence we obtain
\begin{align*}
&\int_{\mathcal{X}_\mathcal{K}(T)} \sup_{s \in C}|\log\zeta(s + i(\log\tau)^{\alpha_j}) - \log\zeta_X(s + i(\log\tau)^{\alpha_j})|\, d\tau \\
&=\int_{\mathcal{X}_\mathcal{K}(T)} \sup_{s \in C}\left|\frac{1}{2\pi i}\int_{\partial{\mathcal{R}}} \frac{\log\zeta(z +i(\log\tau)^{\alpha_j}) -\log\zeta_X(z + i(\log\tau)^{\alpha_j})}{z - s}\, dz \right|\, d\tau \\
&\ll_{C} \int_{\partial{\mathcal{R}}}|dz| \int_{\mathcal{X}_{\mathcal{K}, j}(T)}|\log\zeta(z + i(\log\tau)^{\alpha_j}) -\log\zeta_X(z + i(\log\tau)^{\alpha_j})|\, d\tau \\
&\ll_{\mathcal{R}}  \sup_{\sigma: s \in \partial{\mathcal{R}}} \int_{\mathcal{X}'_\mathcal{K}(T)} |\log\zeta(\sigma + i(\log t)^{\alpha_j}) -\log\zeta_X(\sigma + i(\log t)^{\alpha_j})|\, dt, 
\end{align*}
where $\mathcal{X}'_\mathcal{K}(T) = \{t \in [T/2, 5T/2] : (\log t)^{\alpha_j} \notin \bm{\ell}([(\log (T/2))^{\alpha_j}, (\log({5T}/2))^{\alpha_j}]; \sigma_0, Y) \}$.  

Here, we have 
\begin{align*}
&\int_{\mathcal{X}'_\mathcal{K}(T)} |\log\zeta(\sigma + i(\log t)^{\alpha_j}) -\log\zeta_X(\sigma + i(\log t)^{\alpha_j})|\, dt \\
&\le \int_{\mathcal{X}'_\mathcal{K}(T)} \left|\log\zeta(\sigma +i(\log t)^{\alpha_j}) - \sum_{2 \le n \le Y} \frac{\Lambda(n)}{n^{\sigma + i(\log t)^{\alpha_j}} \log n} \right|\, dt \\
&+\int_{\mathcal{X}'_\mathcal{K}(T)} \left|\sum_{2 \le n \le Y} \frac{\Lambda}{n^{\sigma + i(\log t)^{\alpha_j}} \log n}-\log\zeta_X(\sigma +i(\log t)^{\alpha_j})\right|\, dt.
\end{align*}
By Lemma~\ref{appro}, the first term is bounded by
\begin{align*}
    \frac{1}{T}\int_{\mathcal{X}'_\mathcal{K}(T)} \left|\log\zeta(\sigma +i(\log t)^{\alpha_j}) - \sum_{2 \le n \le Y} \frac{\Lambda(n)}{n^{\sigma + i(\log t)^{\alpha_j}} \log n} \right|\, dt  \ll (\log\log T)^{-1}. 
\end{align*}
Applying the Cauchy--Schwarz inequality, and Lemma~\ref{mean1} to the second term, we obtain 
\begin{align*}
&\frac{1}{T}\int_{\mathcal{X}'_\mathcal{K}(T)} \left|\sum_{2 \le n \le Y} \frac{\Lambda(n)}{n^{\sigma + i(\log t)^{\alpha_j}} \log n}-\log\zeta_X(\sigma +i(\log t)^{\alpha_j})\right|\, dt \\
&\ll \left(\frac{1}{T}\int_{T/2}^{{5T}/2} \left|\sum_{2 \le n \le Y} \frac{\Lambda(n)}{n^{\sigma +i(\log t)^{\alpha_j}} \log n} -\log \zeta_X(\sigma +i(\log t)^{\alpha_j})\right|^2\, dt \right)^{\frac12} \\
&= \left(\frac{1}{T}\int_{T/2}^{{5T}/2} \left|\sum_{X \le n \le Y} \frac{\Lambda(n) \varphi(n/X)}{n^{\sigma +i(\log t)^{\alpha_j}} \log n} \right|^2\, dt \right)^{\frac12} \\
&\ll \left(\sum_{X < n \le Y} \frac{\Lambda(n)^2}{n^{2\sigma_0}(\log n)^2} + (\log T)^{1-\alpha_j} Y^{2-2\sigma_0 + \varepsilon} \right)^\frac12.
\end{align*}

Therefore, 
\begin{align*}
&\limsup_{T \to \infty} \frac{1}{T} \int_{\mathcal{X}_\mathcal{K}(T)} \sup_{s \in C}|\log\zeta(s +i(\log\tau)^{\alpha_j}) -\log\zeta_X(s +i(\log\tau)^{\alpha_j})|\, d\tau \\
&\ll \left( \sum_{X < n}\frac{|\Lambda(n)|^2 }{n^{2\sigma_0} (\log{n})^{2}}\right)^{\frac{1}{2}}
\to 0
\end{align*}
as $X \to \infty$. 

\end{proof}

Let $\mathcal{P}_0$ be a finite set of prime numbers. 
Let $\Omega_{\mathcal{P}_0} = \prod_{p \in \mathcal{P}_0} S_p$ 
and $\Omega^r_{\mathcal{P}_0} = \Omega_{\mathcal{P}_0 1} \times \dots \times \Omega_{\mathcal{P}_0 r}$, 
where $\Omega_{\mathcal{P}_0 j} = \Omega_{\mathcal{P}_0}$ for all $j = 1, \dots, r$.

Define the probability measure on 
$(\Omega^r_{\mathcal{P}_0}, \mathcal{B}(\Omega^r_{\mathcal{P}_0}))$ by 
\[
Q^{\mathcal{P}_0}_T(A) = \frac{1}{T} \mathrm{meas} \left\{\tau \in [T, 2T] : ((p^{-i(\log\tau)^{\alpha_1}})_{p \in \mathcal{P}_0} , \dots, (p^{-i(\log\tau)^{\alpha_r}})_{p \in \mathcal{P}_0}) \in A \right\}, 
\]
for $A \in \mathcal{B}(\Omega^r_{\mathcal{P}_0})$. 
Let $\mathbf{m}^r_{\mathcal{P}_0}$ be the probability Haar measure on $(\Omega^r_{\mathcal{P}_0}, \mathcal{B}(\Omega^r_{\mathcal{P}_0}))$. 
Then, the following lemma holds. 

\begin{lemma} \label{Prob joint}
The probability measure $Q^{\mathcal{P}_0}_T$ converges weakly to $\mathbf{m}^r_{\mathcal{P}_0}$ as $T \to \infty$.  
\end{lemma}

\begin{proof}
    Let $T$ be sufficiently large. 
We consider the Fourier transform $g_T(\underline{\mathbf{n}}_1, \dots, \underline{\mathbf{n}}_r)$ of $Q^{\mathcal{P}_0}_T$ for 
$(\underline{\mathbf{n}}_1, \dots, \underline{\mathbf{n}}_r) = ((n_{1 p})_{p \in \mathcal{P}_0}, \dots (n_{r p})_{p \in \mathcal{P}_0}) \in (\mathbb{Z}^{\#{\mathcal{P}_0}})^r$. 
Its Fourier transform is 
\[
g_T(\underline{\mathbf{n}}_1, \dots, \underline{\mathbf{n}}_r) := \int_{\Omega^r_{\mathcal{P}_0}} \prod_{j=1}^{r} \prod_{p \in \mathcal{P}_0} x^{n_{j p}}_p \,dQ^{\mathcal{P}_0}_T = \frac{1}{T}\int_{T}^{2T} \prod_{j=1}^{r} \prod_{p \in \mathcal{P}_0} p^{-i n_{j p} (\log\tau)^{\alpha_j}} \, d\tau.
\]
It is enough to show that 
\begin{align*}
    \lim_{T \to \infty}  g_T(\underline{\mathbf{n}}_1, \dots, \underline{\mathbf{n}}_r) = 
\begin{cases}
1 & \mbox{if $(\underline{\mathbf{n}}_1, \dots, \underline{\mathbf{n}}_r) = (\bar{0}, \dots, \bar{0})$}, \\
0 & \mbox{if $(\underline{\mathbf{n}}_1, \dots, \underline{\mathbf{n}}_r) \ne(\bar{0}, \dots, \bar{0})$}.
\end{cases}
\end{align*}

If $(\underline{\mathbf{n}}_1, \dots, \underline{\mathbf{n}}_r) = (\bar{0}, \dots, \bar{0})$, it is clear that $g_T(\underline{\mathbf{n}}_1, \dots, \underline{\mathbf{n}}_r) = 1$. 
If $(\underline{\mathbf{n}}_1, \dots, \underline{\mathbf{n}}_r) \ne (\bar{0}, \dots, \bar{0})$, then there exists $j = 1, \dots, r$ such that $\sum_{p \in \mathcal{P}_0}n_{j p} \log{p} \ne 0$ since $\{\log{p} : p \in \mathcal{P}_0 \}$ is linearly independent over the field of rational numbers. 
If $(\underline{\mathbf{n}}_1, \dots, \underline{\mathbf{n}}_r) \ne (\bar{0}, \dots, \bar{0})$, then there exists an index $j = 1, \dots, r$ such that $\sum_{p \in \mathcal{P}_0}n_{j p} \log{p} \ne 0$ since $\{\log{p} : p \in \mathcal{P}_0 \}$ is linearly independent over the field of rational numbers. 
We put $c_j = \sum_{p \in \mathcal{P}_0}n_{j p} \log{p}$ for any $j = 1, \dots, r$, and $j_0 := \max\{1 \le j \le r : c_j \ne 0 \}$. 
Applying the first derivative test, we obtain
\begin{align*}
     \frac{1}{T}\int_{T}^{2T} \prod_{j=1}^{r} \prod_{p \in \mathcal{P}_0} p^{-i n_{j p} (\log\tau)^{\alpha_j}} \, d\tau \ll_{c_j, \alpha_j} \frac{1}{(\log T)^{\alpha_{j_0} -1}} \to 0
\end{align*}
as $T \to \infty$. 
\end{proof}

\begin{proof}[Proof of Proposition~\ref{prop uni}]
    By Portmanteau's theorem (cf. \cite[Theorem~13.16]{Kl}), 
we shall show that $|\mathbb{E}^{\mathcal{Q}_T}[F] - \mathbb{E}^{\mathcal{Q}}[F]| \to 0$ as $T \to \infty$ for all 
bounded Lipschitz continuous function $F$ : $\mathcal{H}^r(\mathcal{R}) \to \mathbb{R}$.
Let $F$ be a bounded real-valued Lipschitz function.
Then there exist nonnegative constants $C_1, C_2$ such that
\[
|F(f)| \le C_1,\ |F(f) - F(g)| \le C_2 d(f, g)
\]
for all $f, g \in \mathcal{H}(\mathcal{R})$.
Now  
\begin{align*}
|\mathbb{E}^{\mathcal{Q}_T}[F] - \mathbb{E}^{\mathcal{Q}}[F]| 
&= |\mathbb{E}^{\mathbb{P}_T}[F(\log\underline{\zeta}(s + i\underline{\log}(\tau)))] - \mathbb{E}^{\m^r}[F(\log\underline{\zeta}(s, X))]| \\
&\le |\mathbb{E}^{\mathbb{P}_T}[F(\log\underline{\zeta}(s + i\underline{\log}(\tau)))] - \mathbb{E}^{\mathbb{P}_T}[F(\log\underline{\zeta}_X(s + i\underline{\log}(\tau)))]| \\
&\ + |\mathbb{E}^{\mathbb{P}_T}[F(\log\underline{\zeta}_X(s + i\underline{\log}(\tau)))] - \mathbb{E}^{\m^r}[F(\log\underline{\zeta}_X(s, X))]| \\
&\ + |\mathbb{E}^{\m^r}[F(\log\underline{\zeta}_X(s, X))] - \mathbb{E}^{\m^r}[F(\log\underline{\zeta}(s, X))]| \\
&=: E_1  +E_2 + E_3
\end{align*}
holds for $X \ge 2$. 

First, we evaluate $E_1$. 
Using the definition of $F$, one can see that  
\begin{align*}
    E_1 
    &\le \frac{1}{T}\int_T^{2T} \left|F(\log\underline{\zeta}(s + i\underline{\log}(\tau))) - F(\log\underline{\zeta}_X(s + i\underline{\log}(\tau)))\right|d\tau \\
    &\le \frac{1}{T}\int_{\mathcal{\mathcal{X}_\mathcal{K}}(T)}\left|F(\log\underline{\zeta}(s + i\underline{\log}(\tau))) - F(\log\underline{\zeta}_X(s + i\underline{\log}(\tau)))\right|d\tau \\
    &+\frac{1}{T}\int_{[T, 2T] \setminus \mathcal{\mathcal{X}_\mathcal{K}}(T)}\left|F(\log\underline{\zeta}(s + i\underline{\log}(\tau))) - F(\log\underline{\zeta}_X(s + i\underline{\log}(\tau)))\right|d\tau \\
    &\le \frac{C_2(F)}{T}\int_{\mathcal{X}_{\mathcal{K}}(T)} \underline{d}(\log\underline{\zeta}(s + i\underline{\log}(\tau)), \log\underline{\zeta}_X(s + i\underline{\log}(\tau)))d\tau \\
    &+ 2C_1(F) \frac{\meas([T, 2T]\setminus \mathcal{X}_{\mathcal{K}}(T))}{T}
\end{align*}
By Lemma~\ref{appro X} and the definition of $\underline{d}$, it follows that 
\begin{align*}
  \lim_{X \to \infty}  \limsup_{T \to \infty} \frac{C_2(F)}{T}\int_{\mathcal{X}_{\mathcal{K}}(T)} \underline{d}(\log\underline{\zeta}(s + i\underline{\log}(\tau)), \log\underline{\zeta}_X(s + i\underline{\log}(\tau)))d\tau =0. 
\end{align*}
From $\meas(\mathcal{X}_{\mathcal{K}}(T)) \sim T$ as $T \to \infty$, we have 
\[
\lim_{T \to \infty} \frac{\meas([T, 2T]\setminus \mathcal{X}_{\mathcal{K}}(T))}{T} =0. 
\]

Lemma~\ref{Prob joint} implies $E_2 \to 0$ as $T \to \infty$. 
In addition, by \cite[Lemma~5]{En}, $\lim_{X \to \infty}E_3 = 0$. 
This completes the proof.
\end{proof}

\section{Proof of Theorem~\ref{main1} and Theorem~\ref{main2}}

\subsection{Proof of Theorem~\ref{main1}}

For a probability space $(\Omega, \mathcal{M}, P)$, a minimal closed set $S \subset \Omega$ such that $P(S) = 1$ is called a support of $P$. 
It is known that the support of $\log\zeta(s, \omega)$ coincides with $\mathcal{H}(\mathcal{R})$ (cf. \cite[Proposition~2]{En}). 

\begin{proof}[Proof of theorem~\ref{main1}]
Let $f_1(s), \dots, f_r(s)$ be continuous functions on $\mathcal{K}$ and holomorphic on the interior of $\mathcal{K}$. 
Fix $\varepsilon >0$. 
By Mergelyan's theorem, there exist polynomials $P_1(s), \dots, P_r(s)$ such that
\[
\max_{1 \le j \le r} \sup_{s \in \mathcal{K}} |f_j(s) - P_j(s)| < \varepsilon.
\]
Define an open set of $\mathcal{H}^r(\mathcal{R})$ by 
\[
\Phi_r = \Phi(P_1, \dots, P_r) := \left\{(g_1(s), \dots, g_r(s)) \in \mathcal{H}^r(\mathcal{R}) : \max_{1 \le j \le r}\sup_{s \in \mathcal{K}} |g_j(s) - P_j(s)| < \varepsilon \right\}.
\] 
Recall that
\begin{align*}
    \mathcal{Q}(\Phi_r) = \m^r\left\{(X_j)_{1 \le j \le r} \in \Omega^r : \max_{1 \le j \le r} \sup_{s \in \mathcal{K}} \left|\log\zeta(s, X_j) - P_j(s) \right| < \varepsilon \right\}. 
\end{align*}
Since $\log\zeta(s, X_j)$ does not depend on the choice of $X_k \in \Omega_k$ for any $k \ne j$, 
it follows that
\begin{align*}
     \mathcal{Q}(\Phi_r) = \prod_{j=1}^r \m\left\{X_j \in \Omega_j : \sup_{s \in \mathcal{K}} \left|\log\zeta(s, X_j) - P_j(s) \right| < \varepsilon \right\}. 
\end{align*}
Combining these arguments and Portmanteau's theorem (cf. \cite[Theorem~13.16]{Kl}), we have 
\begin{align*}
&\liminf_{T \to \infty} \frac{1}{T} \meas \left\{\tau \in [T, 2T] : \max_{1 \le j \le r} \sup_{s \in \mathcal{K}} |\log\zeta(s + i(\log \tau)^{\alpha_j}) - f_j(s)| < \varepsilon \right\} \\ &= \liminf_{T \to \infty} \mathcal{Q}_T(\Phi_r) \\
&\ge \mathcal{Q}(\Phi_r) = \prod_{j=1}^r \m\left\{X_j \in \Omega_j : \sup_{s \in \mathcal{K}} \left|\log\zeta(s, X_j) - P_j(s) \right| < \varepsilon \right\} > 0. 
\end{align*}
Since 
\begin{align*}
    &\max_{1 \le j \le r}\sup_{s \in \mathcal{K}} |\log\zeta(s + i(\log \tau)^{\alpha_j}) -f_j(s)| \\
&\le \max_{1 \le j \le r} \sup_{s \in \mathcal{K}} |\log\zeta(s + i(\log \tau)^{\alpha_j}) - P_j(s)| + \max_{1 \le j \le r} \sup_{s \in \mathcal{K}} |P_j(s) - f_j(s)|, 
\end{align*}
we get the first assertion. 

Next, we show that under the Riemann hypothesis we can extend this universality to $1/2 < \sigma < 1$. 
The restriction $x_\Phi(\alpha_1) < \sigma < 1$ arises from Lemma~\ref{meas log}. 
Assuming the Riemann hypothesis, we have $\Phi(\sigma) = 0$. 
Therefore, we do not have to consider Lemma~\ref{meas log}. 
Hence, we can extend joint universality for the region $1/2 < \sigma < 1$. 
\end{proof}

\subsection{Proof of Theorem~\ref{main2}}

The proof of Theorem~\ref{main2} is based on the proof of Theorem~\ref{main general}. 
We fix $r > 1$. 
The function $\gamma(t) = (\log t)^r$ satisfies (F1) and (F2). 
Since
\[
T\tilde\gamma(T)=\frac{r(\log2T)^{r-1}}2,
\]
all results in Section~\ref{section3} remain valid. Therefore, it suffices to establish analogues of Lemma~\ref{measure error}, Lemma~\ref{higher general} and Proposition~\ref{W general} in the case $(\log t)^r$. 

Fix $\sigma > x_\Phi(r)$. 
We put $\sigma_0 = \frac12\left(x_\Phi(r) + \sigma \right)$, 
$Y = (\log\log T)^{\frac{14}{\sigma_0 - \frac12}}$, and 
\[
\mathcal{I}_{r}(T) = \{t \in [T, 2T] : (\log t)^r \notin \bm{\ell}\left([(\log T/2)^r, (\log 5T/2)^r ];Y\right)\}. 
\]

\begin{lemma} \label{measure error log}
    For sufficiently small $\varepsilon > 0$, 
    \[
    \frac{1}{T}\meas\{[T, 2T] \setminus \mathcal{I}_r(T)\} \ll_\varepsilon  Y (\log T)^{1 -r+  r\Phi(\sigma_0)}. 
    \]
    Consequently, we have $\mathcal{I}_r(T) \sim T$. 
\end{lemma}

\begin{proof}
    It suffices to take $\mathcal{K} = \{\sigma\}$ in Lemma~\ref{meas log}. 
\end{proof}

Using Lemma~\ref{measure error log}, we obtain higher moments of the Riemann zeta-function with $(\log t)^r$. 

\begin{lemma} \label{higher log}
        Let $T$ be sufficiently large. 
    Then, for $k \ge 1$ with \eqref{Y}, 
    \begin{align*}
        &\frac{1}{T}\int_{\mathcal{I}_{r}(T)} |\log\zeta(\sigma + i(\log t)^r)|^{2k}dt \\
        &\ll A^k\frac{k^{2k(1-\sigma_0)}}{(\log2k)^{2k\sigma_0}} + A^k(\log \log T)^{-k}, 
    \end{align*}
    where $A$ is a positive constant. 
\end{lemma}

\begin{proof}
    By using Lemma~\ref{measure error log}, we can prove this lemma in the same way as Lemma~\ref{higher general}. 
\end{proof}

Additionally, we can show the following proposition. 

\begin{proposition} \label{W log}
    Let $L = c_3((r-1)\log\log T)^{\sigma}$, where $c_3$ is a small positive constant. 
    Then, 
    \[
    \frac{1}{T} \int_T^{2T} W_{L, \mathcal{R}}(P_Y(\sigma + i(\log t)^r))dt = \mathbb{E}[W_{L, \mathcal{R}}(\log\zeta(\sigma, X))] + O\left((\log\log T)^{-2} \right)
    \]
    for $w = u + iv$. 
    
\end{proposition}

\begin{proof}

The proposition follows by repeating the proof of Proposition~\ref{W general}. 

\end{proof}

\begin{proof}[Proof of Theorem~\ref{main2}]
Let $T$ be sufficiently large. 
Put $M = \log\log\log T$, and define $\mathcal S_M$ as in Section~\ref{proof main general}. 
Then, by the same arguments as in Section~\ref{proof main general}, we have
\[
D_{\sigma, r}(T) \ll D_{\sigma, r, Y}(T) + \varepsilon M + (\log\log T)^{-2}, 
\]
where 
\[
D_{\sigma, r, Y}(T) = \sup_{\mathcal{R} \in \mathcal{S}_{2M}} \left|\mathbb{P}_{\mathcal{I}_{r}(T)}(P_Y(\sigma + i(\log t)^r) \in \mathcal{R}) - \mathbb{P}(\log\zeta(\sigma, X) \in \mathcal{R} )\right|, 
\]
and $\varepsilon = (\log\log T)^{-4}$. 
Applying Proposition~\ref{W log} together with the arguments in Section~\ref{proof main general} to estimate $D_{\sigma, r, Y}(T)$, we obtain 
\[
D_{\sigma, r, Y}(T) \ll (\log\log T)^{-2} + L^{-1} \ll ((r-1)\log\log T)^{-\sigma}. 
\]
Thus, we obtain the first assertion. 

Under the Riemann hypothesis, Lemma~\ref{measure error log} is no longer necessary.  
Therefore, we can extend this estimate for $1/2 < \sigma < 1$. 
\end{proof}

\section{Concluding remarks} \label{conclusion}

First, we note that $\mathcal{F}'$ is a proper subset of $\mathcal{F}$. 

\begin{proposition} \label{remark}
    There exists $\gamma \in \mathcal{F}$ such that $\gamma \notin\mathcal{F}'$, i.e. $\mathcal{F}' \subsetneq \mathcal{F}$. 
\end{proposition}

\begin{proof}
    We construct a function $\gamma$ satisfying (F1), (F2), and (F3) but not satisfying (F4). 
    Observe that (F4) is trivially satisfied when $c=1$. 
    Let $c > 1$ and $\varepsilon > 0$.
    Denote $T_0 = 1$, $A_0 := \gamma(1) := 1$, $m_0 := \gamma'(1) := 1$. 
    We define $T_n = (2c)^n$, and define $A_n := \gamma(T_n)$, and $m_n = \gamma'(T_n)$ inductively as follows.
    For $t \in [T_n, cT_n]$, define
    \[
    \gamma'(t) = m_n + \frac{M_n - m_n}{(c-1)T_n}(t - T_n), 
    \]
    where 
    \[
    M_n := \max\left\{m_n + \frac{2\exp(nA_n^{\varepsilon})}{(c-1)T_n}, 1 \right\}, 
    \]
    and for $t \in [cT_n, T_{n+1}]$, set  $\gamma'(t) = M_n$. 
    Then, for $t \in [T_n, cT_n]$, 
    \begin{align} \label{gamma c}
    \gamma(cT_n) = A_n + \frac12(c-1)(M_n + m_n)T_n \ge \exp\left(n A_n^{\varepsilon} \right)
    \end{align}
    holds. 
    Under this setting, we can define 
    \[
    A_{n+1} = \gamma(T_{n+1}) = \gamma(cT_n) + cT_nM_n, \quad m_{n+1} = \gamma'(T_{n+1}) = M_{n}. 
    \]
   By construction, $\gamma'$ is increasing, so $\gamma$ satisfies (F1) and (F2). 
   Moreover, from the following inequality:
    \begin{align*}
        \gamma(T) = 1 + \int_{1}^{T} \gamma'(t)dt \le 1 +  (T-1)\gamma'(T) \le T\gamma'(T), 
    \end{align*}
    $\gamma$ satisfies (F3). 
    Hence, $\gamma \in \mathcal{F}$. 
    However, by \eqref{gamma c}, 
    \begin{align*}
    \frac{\log\gamma(cT_n)}{\gamma(T_n)^\varepsilon} \ge n \to \infty
    \end{align*}
    as $n \to \infty$. 
    Thus, $\gamma$ does not satisfy (F4). 
\end{proof}

Next, we discuss the logarithmic shift $\gamma(t)=\log t$. 
In a proof of universality theorem with a general shift, the following limit is one of the key ingredients:
\begin{align*}
\lim_{T \to \infty} \frac{1}{T}\int_{T_0}^T e^{ic\gamma(t)}dt = 
 \begin{cases}
    1 & \text{if $c = 0 $,} \\
    0                 & \text{if $c \ne 0$,} 
  \end{cases}
\end{align*}
where $\gamma$ is a real-valued function, and $c \in \mathbb{R}$. 
The case $c = 0$ is trivial. 
Therefore, we have to consider the case $c \ne 0$. 
For the logarithmic function $\gamma(t) = \log t$, 
\begin{equation} \label{log int}
\frac{1}{T}\int_{T_0}^T e^{ic\log t}dt = \frac{1}{ic + 1}\left(T^{ic} - T_0^{ic}\right). 
\end{equation}
In general, \eqref{log int} does not have a limit value. 
Thus, this argument does not determine whether universality holds for the logarithmic shift. 
Furthermore, the proof of Theorem~\ref{main1} relies on the admissible range of $Y$ in Proposition~\ref{prop1 general}. 
Indeed, $\log t$ satisfies (F1) and (F2). 
Hence, Lemma~\ref{higher moment general} is valid for $\log t$. 
However, taking $\gamma(t) = \log t$, we see that $T\tilde{\gamma}(T) = 1/2$. 
In other words, we are forced to take $Y$ to be a constant. 
Consequently, we can not determine the value
\[
 \frac{1}{T} \meas\{t \in [1, T] : \log \zeta(\sigma + i\log t) \in \mathcal{R}\}, \quad (T\to\infty)
\]
by this method. 
Therefore, establishing a limit theorem for the logarithmic shift would require some fundamentally different approach.

\subsection*{Acknowledgments} 
The author would like to thank to professor Kohji Matsumoto for his comments. 
This work was financially supported by JSPS Research Fellow (Grant Number:25KJ1412).

\begin{flushleft}
{\footnotesize
{\sc
Graduate School of Mathematics, Nagoya University, Chikusa-ku, Nagoya 464-8602, Japan.
}\\
{\it E-mail address}: {\tt m21029d@math.nagoya-u.ac.jp}
}
\end{flushleft}

\end{document}